\pdfoutput=1
\documentclass{article}
\usepackage{caption}
\usepackage{subcaption}
\usepackage{tikz}
\usepackage[latin1]{inputenc} 
\usepackage{amssymb}
\usepackage{amsfonts}
\usepackage{epsf,epsfig,amsmath,amssymb,amsfonts,latexsym}
\usepackage{enumerate}
\usepackage{multicol}
\usepackage{graphicx}
\usepackage{stmaryrd}
\usetikzlibrary{arrows,decorations.markings, calc, matrix, shapes,automata,fit,patterns,decorations.pathreplacing,calligraphy,fadings}
\usepackage{color}
\usepackage[normalem]{ulem}
\usepackage{diagbox}
\usepackage{amsthm}
\usepackage{fancyhdr}
\usepackage{verbatim}
\usepackage{fullpage}
\usepackage{hyperref}
\usepackage{cleveref}
\usepackage{mathtools}

\allowdisplaybreaks

\newtheorem{theorem}{Theorem}[section]
\newtheorem{lemma}[theorem]{Lemma}
\newtheorem{claim}[theorem]{Claim}
\newtheorem{proposition}[theorem]{Proposition}
\newtheorem{corollary}[theorem]{Corollary}
\newtheorem{maintheorem}{Theorem}
\theoremstyle{definition}
\newtheorem{definition}[theorem]{Definition}
\newtheorem{example}[theorem]{Example}
\newtheorem{remark}[theorem]{Remark}
\newtheorem{observation}[theorem]{Observation}
\newtheorem{question}[theorem]{Question}

\newcommand{\yolk}{
\foreach \i in {0,...,5}{
            \draw[fill = white!30] (\i,5) rectangle (\i+1,6);
            \node at (\i+0.5,5.5) {$\mathtt{1}$};
            \draw[fill = white!30] (\i,0) rectangle (\i+1,1); 
            \node at (\i+0.5,0.5) {$\mathtt{1}$};
            \draw[fill = white!30] (5,\i) rectangle (6,\i+1); 
            \node at (5.5,\i+0.5) {$\mathtt{1}$};
            \draw[fill = white!30] (0,\i) rectangle (1,\i+1); 
            \node at (0.5,\i+0.5) {$\mathtt{1}$};
            }
            \foreach \i in {1,...,4}{
            \draw[fill = white!30] (\i,4) rectangle (\i+1,5); 
            \node at (\i+0.5,4.5) {$\mathtt{0}$};
            \draw[fill = white!30] (\i,1) rectangle (\i+1,2); 
            \node at (\i+0.5,1.5) {$\mathtt{0}$};
            \draw[fill = white!30] (4,\i) rectangle (5,\i+1); 
            \node at (4.5,\i+0.5) {$\mathtt{0}$};
            \draw[fill = white!30] (1,\i) rectangle (2,\i+1); 
            \node at (1.5,\i+0.5) {$\mathtt{0}$};
            }
}
\def\fix{\operatorname{Fix}}
\def\stab{\operatorname{Stab}}
\def \NN{\mathbb N}
\def \ZZ{\mathbb Z}

\def\N{\mathbb N}
\def\Z{\mathbb Z}

\newcommand{\diam}{\operatorname{diam}}
\newcommand{\define}[1]{\textbf{#1}}
\def\F{\mathcal F}

\def\aut{\operatorname{Aut}}
\def\Aut{\operatorname{Aut}}
\def\Id{\mathbf{1}}
\newcommand*\quot[2]{{^{\textstyle #1}\big/_{\textstyle #2}}}

\def\L{L}

\title{The automorphism group of a strongly irreducible subshift on a group}

\author{
	Sebasti\'an Barbieri, Nicanor Carrasco-Vargas and Paola Rivera-Burgos
}

\newcommand{\Addresses}{{
		\bigskip
		
		\hskip-\parindent   S.~Barbieri, \textsc{Departamento de Matem\'{a}tica y ciencia de la computaci\'{o}n, Universidad de Santiago de Chile, Santiago, Chile.}\par\nopagebreak
		\textit{E-mail address}: \texttt{sebastian.barbieri@usach.cl}
		
		\medskip
		
		\hskip-\parindent   N.~Carrasco-Vargas, \textsc{
        Mathematics faculty, Jagiellonian University, Krak\'ow, Poland.}\par\nopagebreak
		\textit{E-mail address}: \texttt{nicanor.vargas@uj.edu.pl}
		
		\medskip
		
		\hskip-\parindent   P.~Rivera-Burgos, \textsc{Departamento de Matem\'{a}tica y ciencia de la computaci\'{o}n, Universidad de Santiago de Chile, Santiago, Chile.}\par\nopagebreak
		\textit{E-mail address}: \texttt{paola.rivera.b@usach.cl}
}}

\begin{document}

\maketitle

\begin{abstract}
        We study the automorphism group $\operatorname{Aut}(X)$ of a non-trivial strongly irreducible subshift $X$ on an arbitrary infinite group $G$ and generalize classical results of Ryan, Kim and Roush.
        
        We generalize Ryan's theorem by showing that the center of $\operatorname{Aut}(X)$ is generated by shifts by elements of the center of $G$ modded out by the kernel of the shift action.

       We generalize Kim and Roush's theorem by showing that if the free group $F_k$ of rank $k\geq 1$ embeds into $G$, then the automorphism group of any full $F_k$-shift embeds into $\operatorname{Aut}(X)$. If $X$ is an SFT, or more generally, if $X$ satisfies the strong topological Markov property, then we can weaken the conditions on $G$. In this case we show that the automorphism group of any full $\mathbb{Z}$-shift embeds into $\operatorname{Aut}(X)$ provided $G$ is not locally finite, and that the automorphism group of any full $F_k$-shift embeds into $\operatorname{Aut}(X)$ whenever $G$ is nonamenable. Our results rely on a new marker lemma which is valid for any nonempty strongly irreducible subshift on an infinite group.

        We remark that our results are new even for  $G=\mathbb{Z}$ as they do not require the subshift to be an SFT. 

		\medskip
		
		
		\smallskip
		
	\end{abstract}

\section{Introduction}

The study of the automorphism group of a full $\ZZ$-shift goes back to the work of Hedlund~\cite{Hedlund1969} in the sixties, where it was shown that this group is countable and contains copies of every finite group. A few years later Ryan \cite{Ryan1972} computed its center and proved that it is generated by the shift map. 


These and other initial results were generalized to non-trivial mixing $\ZZ$-subshifts of finite type (SFT) by Boyle, Lind and Rudolph~\cite{BoyleLindRudolph} through a refinement of the ``marker technique'' of Krieger~\cite{Krieger1982} (see also~\cite{Boyle93}). They showed that under these assumptions the automorphism group is residually finite and contains isomorphic copies of several kinds of groups, including non-abelian free groups and the direct sum of countably many copies of $\ZZ$. Kim and Roush~\cite{KimRoush1990} expanded on these results by employing the ``conveyor belt'' technique. They proved that the automorphism group of the full $\ZZ$-shift embeds into the automorphism group of any non-trivial mixing $\ZZ$-SFT.

A central  open problem in this context is that of distinguishing automorphism groups of full $\ZZ$-shifts on distinct alphabets: it is not known whether the automorphism groups of $\{0,1\}^\ZZ$ and $\{0,1,2\}^\ZZ$ are isomorphic. However, they embed into one another by the result of Kim and Roush, and thus they can not be distinguished by their subgroups. In contrast, Ryan's theorem can be used to distinguish the automorphism groups of some pairs of full $\ZZ$-shifts, such as $\{0,1\}^\ZZ$ and $\{0,1,2,3\}^\ZZ$ (see~\cite[Section 4]{BoyleLindRudolph} and \cite{hartman2022stabilized}).

For $\ZZ^d$-subshifts much less is known. One of the first results in this direction is due to Ward~\cite{Ward1994}, who proved that the automorphism group of any non-trivial strongly irreducible $\ZZ^d$-SFT contains isomorphic copies of any finite group. More recently, Hochman~\cite{Hochman2010} generalized Ward's result to $\ZZ^d$-SFTs with positive entropy. Furthermore, he showed that if such a subshift has dense periodic orbits, then the automorphism group of any full $\ZZ$-shift embeds into it. He also showed that the natural generalization of Ryan's theorem holds for transitive $\ZZ^d$-SFTs with positive entropy. 

The purpose of this work is to generalize the classical results of Ryan, Kim and Roush in two different directions. First, we consider strongly irreducible subshifts (also known as subshifts with strong specification) on arbitrary infinite groups, and second, we either remove the SFT hypothesis, or we replace it by a much weaker condition, the strong topological Markov property (strong TMP).


Our first result is a broad generalization of Ryan's theorem. For a group $G$, denote its center by $Z(G)$ and for a $G$-subshift $X$, denote by $\operatorname{Fix}(X) = \{ g \in G: gx=x \mbox{ for all } x \in X\}$ the kernel of the shift action.

\begin{maintheorem}\label{thm:Ryan}
    Let $G$ be an infinite group and $X$ be a nonempty strongly irreducible $G$-subshift. Then $Z(\Aut(X))$ is generated by shifts by elements $g\in G$ with $g\fix(X) \in Z(G/\fix(X))$. In particular \[ Z(\Aut(X)) \cong Z\left(\quot{G}{\operatorname{Fix}(X)}\right). \]
\end{maintheorem}

We note that~\Cref{thm:Ryan} fails in the case where $G$ is a finite group. See~\Cref{ex:finite_ryan}. In the case where $G$ has no non-trivial normal finite subgroups, for instance, if $G$ is torsion-free or if $G$ is infinite and simple, then the shift action on a non-trivial strongly irreducible subshift $X$ is faithful (\Cref{lem:if-X-SI-then-Fix-contained-in-K}) and thus the conclusion of~\Cref{thm:Ryan} can be read simply as $Z(\Aut(X)) \cong Z(G)$.

A direct consequence of~\Cref{thm:Ryan} is that whenever $G,H$ are groups with non-isomorphic centers, then $\Aut(\{0,1\}^{G})$ and $\Aut(\{0,1\}^{H})$ are not isomorphic (see~\Cref{cor:hochmananswer} for an even more general statement). In particular, this answers negatively a question of Hochman~\cite{Hochquestion} on whether $\Aut(\{0,1\}^{\ZZ^n})$ is isomorphic to $\Aut(\{0,1\}^{\ZZ^m})$ for distinct $n,m$.

As mentioned before, one of the most famous applications of Ryan's theorem is a proof that $\Aut(\{0,1\}^{\ZZ})$ and $\Aut(\{0,1,2,3\}^{\ZZ})$ are non-isomorphic~\cite{BoyleLindRudolph}. This technique has been generalized to distinguish automorphism groups where the alphabets are different powers of primes using the dimension representation ~\cite[Theorem 2.5]{hartman2022stabilized}. Using~\Cref{thm:Ryan} along with a technique of Lind~\cite[Theorem 8]{Lind1984} and Lind's classification of topological entropies of mixing $\ZZ$-SFTs, we obtain the following general result.

\begin{maintheorem}\label{thm:roots_of_shift_ingroups}
    Let $G$ be an infinite group and suppose there exists an epimorphism $\psi\colon G \to \ZZ$ such that $\psi(Z(G))= \ZZ$. For every integer $n\geq 2$ and positive integers $k,\ell$ we have \[ \Aut( \{1,\dots,n^k\}^G) \cong \Aut( \{1,\dots,n^{\ell}\}^G) \mbox{ if and only if } k=\ell.   \]
\end{maintheorem}

We remark that the condition on~\Cref{thm:roots_of_shift_ingroups} is satisfied by any group which is of the form $G = H \times \ZZ$ for some group $H$. In particular, it holds for any infinite and finitely generated abelian group.

Next we provide two generalizations of the result of Kim and Roush.

\begin{maintheorem}\label{thm:EmbeddingZ}
    Let $G$ be a group and $X$ be a non-trivial strongly irreducible $G$-subshift. For every finite alphabet $A$: \begin{enumerate}[$\mathrm{(}$1$\mathrm{)}$]
        \item If $G$ admits a torsion-free element, then $\Aut(A^{\ZZ})$ embeds into $\Aut(X)$.
        \item If $F_2$ embeds into $G$, then $\Aut(A^{F_k})$ embeds into $\Aut(X)$ for every $k \geq 1$.
    \end{enumerate}
\end{maintheorem}
In the case $G=\ZZ$ it follows from this result that $\Aut(A^{\ZZ})$ embeds into $\Aut(X)$ for every nontrivial strongly irreducible $\ZZ$-subshift $X$.  Since every mixing $\ZZ$-SFT is in particular strongly irreducible, we recover and strengthen the original result of Kim and Roush. 

In the case $G=\ZZ^d$, $d>1$, the conclusion of our result is again that $\Aut(A^{\ZZ})$ embeds into $\Aut(X)$ for every nontrivial strongly irreducible $\ZZ^d$-subshift $X$. In this case our result is a priori not comparable with Hochman's result \cite[Theorem 1.3]{Hochman2010}, which shows that $\Aut(A^\ZZ)$ embeds into $\Aut(X)$ whenever $X$ is a $\ZZ^d$-SFT with positive entropy and a dense set of periodic points. To our knowledge, it is still unknown whether there exist strongly irreducible $\ZZ^d$-SFTs without dense periodic orbits for $d \geq 3$.


For many groups, such as infinite torsion groups and counterexamples to the Von Neumann conjecture, the hypotheses of~\Cref{thm:EmbeddingZ} do not hold. In this setting, we can establish an embedding result under the additional hypothesis of the strong TMP.

\begin{maintheorem}\label{thm:EmbeddingF2}
    Let $G$ be a group and $X$ be a non-trivial strongly irreducible $G$-subshift which satisfies the strong TMP. For every finite alphabet $A$: \begin{enumerate}[$\mathrm{(}$1$\mathrm{)}$]
        \item If $G$ is not locally finite, then $\Aut(A^{\ZZ})$ embeds into $\Aut(X)$.
        \item If $G$ is nonamenable, then $\Aut(A^{F_k})$ embeds into $\Aut(X)$ for every $k \geq 1$.
    \end{enumerate}
\end{maintheorem}
The condition that $G$ is not locally finite is necessary.  Indeed, a direct consequence of the Curtis-Hedlund-Lyndon theorem is that in this case every element of $\Aut(A^{G})$ has torsion, and therefore neither $\ZZ$ nor $\Aut(A^{\ZZ})$ embed into $\Aut(A^{G})$ for $|A|\geq 2$. In fact, the automorphism group $\Aut(A^{G})$ is locally finite, see~\cite{raymond2023shifts}. Regarding the second statement, we do not know if the nonamenability condition is necessary. In fact, we do not know whether $\Aut(\{0,1\}^{F_2})$ embeds into $\Aut(\{0,1\}^{\ZZ})$. 

Next we highlight two particular cases of \Cref{thm:EmbeddingF2}. Consider two finite alphabets $A,B$ with cardinality at least $2$. First, $\Aut(A^{\ZZ})$ embeds into $\Aut(B^G)$ for any non-locally finite group $G$. In particular, all of the subgroup realization results for $\Aut(A^\ZZ)$, such as those of Boyle, Lind and Rudolph~\cite{BoyleLindRudolph}, are also valid for $\Aut(B^G)$ whenever $G$ is not locally finite.  Second, $\Aut(A^{F_k})$ embeds into $\Aut(B^{F_d})$ for every $d,k\geq 2$. In particular, the subgroup structure of the automorphism group of a non-trivial full shift on a non-abelian free group is universal, that is, it does not depend upon the alphabet or the rank.

We remark that our methods provide a direct proof of the following: if $G$ is an infinite group and $X$ is a nontrivial strongly irreducible $G$-subshift, then every finite group embeds into $\Aut(X)$ (\Cref{cor:generalized_ward}). This strengthens Ward's result~\cite{Ward1994} about strongly irreducible SFTs on $\ZZ^d$. Notice that this result is not a direct consequence of~\Cref{thm:EmbeddingF2}, as it is valid even on infinite locally finite groups.




The proofs of Theorems~\ref{thm:Ryan}, \ref{thm:EmbeddingZ} and \ref{thm:EmbeddingF2} rely on a generalized marker lemma which is valid for any nonempty strongly irreducible subshift on an infinite group. We remark that ``marker lemmas'' for general groups have been established in many instances in the literature, but they often mean different things, see for instance~\cite{McGoffPavlov2021,meyerovitch2024embeddingtheoremmultidimensionalsubshifts,Hochman2010}. Our marker lemma is in spirit a generalization of the notion used in~\cite{Ward1994,BarMey2023,Hochman2010}, but we factor into our definition the possibility that the shift action might not be faithful.

Let us be more precise: let $G$ be a group, $F$ a finite subset of $G$ and $g \in G$. A pattern $p$ with support $F$ is said to have a $g$-overlap if $p$ and its translate by $g$ coincide in $F\cap gF$. Given a $G$-subshift $X$ and two finite subsets $Y\subset W$ of $G$, a $(Y,W)$-{marker} for $X$ is a pattern $p$ which occurs in $X$, has support $W\smallsetminus Y$, and it does not have $g$-overlaps for $g\in WY^{-1}\smallsetminus \operatorname{Fix}(X)$. We say that $X$ is \define{markable} when for every finite $Y\subset G$  there exists a finite $W\subset G$ and a $(Y,W)$-marker for $X$. 

\begin{maintheorem}\label{thm:markers}
Let $G$ be an infinite group. Every nonempty strongly irreducible $G$-subshift is markable. Furthermore, if $G$ is finitely generated and endowed with a word metric, then $X$ admits a $(B(r),B(37r))$-marker for all large enough $r$. 
\end{maintheorem}

One can think about a $(Y,W)$-marker $p$ as a sort of ``protective bubble'' around $Y$. More precisely, if $p$ occurs in some configuration $x$, and a shifted copy of $p$ also occurs at some $g \in G$, then, unless $g\in \fix(X)$, the support of $p$ must be disjoint from $gY$ . That is, $gY \cap W = \varnothing$. We illustrate this in~\Cref{fig:buevitos}.

We remark that the proof of Theorems~\ref{thm:Ryan} and~\ref{thm:EmbeddingZ} use the fact that every nonempty strongly irreducible $G$-subshift is markable, but do not require the explicit bound in the second part of the statement. This bound is crucial for the proof of~\Cref{thm:EmbeddingF2}.

This paper is organized as follows. In~\Cref{sec:prelim} we introduce general notation and notions on symbolic dynamics such the strong TMP. We also prove the necessary fundamental results on strongly irreducible subshifts and automorphism groups. In~\Cref{sec:markers} we introduce our notion of markers and show~\Cref{thm:markers}. Next, in~\Cref{sec:buebito}, we use the notion of markers to construct a collection of patterns which we call ``egg markers''. These are collections of patterns that are built around a common marker, but also carry some local information (``yolk'') which can be freely exchanged. These collections of egg markers are the main tool used to construct automorphisms in a strongly irreducible subshift. In~\Cref{sec:ryan} we give the proof of~\Cref{thm:Ryan} and its main application,~\Cref{thm:roots_of_shift_ingroups}. Finally, in~\Cref{sec:embedding} we begin with a toy result meant to introduce the reader to the technique of conveyor belts adapted to groups, and then we prove both embedding results, Theorems~\ref{thm:EmbeddingZ} and~\ref{thm:EmbeddingF2}. Finally, in~\Cref{sec:questions} we give a few perspectives on how these results might be generalized and state a few open problems.

\subsubsection*{Acknowledgments}
The authors are grateful to Ville Salo who suggested the proof of~\Cref{lem:existence_egg_markers_general}, which in turn allowed us to get rid of a hypothesis on both~\Cref{thm:Ryan} and~\Cref{thm:EmbeddingZ}. We also acknowledge that this project was partly developed while attending CIRM's thematic month ``Discrete Mathematics \& Computer Science: Groups, Dynamics, Complexity, Words''. S. Barbieri was supported by ANID 1240085 FONDECYT regular, AMSUD240026 and ECOS230003. N. Carrasco-Vargas was supported by a grant from the Priority Research Area SciMat under the Strategic Programme Excellence Initiative at Jagiellonian University. P. Rivera-Burgos was supported by ANID 21212340 doctorado nacional, Vicerrector\'ia de Postgrado de la Universidad de Santiago de Chile, and MathAMsud 210020 Dynamical Group Theory.

\section{Preliminaries}\label{sec:prelim}


Throughout the article, $G$ denotes a group endowed with the discrete topology and with identity element $\Id_G$. We will use the notation $F\Subset G$ to denote that $F$ is a finite subset of $G$, and $\langle F\rangle$ to represent the finitely generated subgroup of $G$ generated by $F$. Given $g\in G$ and $F,K\subset G$, we write $F^{-1}=\{g : g^{-1}\in F\}$, $gF=\{gh : h\in F\}$ and $FK=\{gh : g\in F,\ h\in K\}$. We say that $F$ is symmetric if $F=F^{-1}$. 

Given two groups $G$ and $H$, we write $G\cong H$ to say that they are isomorphic. The center of $G$ is the normal subgroup $Z(G)= \{g \in G: gh=hg \mbox{ for every } h \in G\}$. For a subgroup $H \leqslant G$ we denote its left cosets in $G$ by $G/H$. For a set $A$ we denote by $\operatorname{Sym}(A)$ its group of permutations.

Let $S\Subset G$ be a finite and symmetric generating set of some group $G$. Given $g\in G$ we denote by $|g|_S$ the word length of $g$ with respect to the generating set $S$ and note by $d_S(g,h)=|g^{-1}h|_S$ the associated word metric. That is, $|g|_S$ is the length of the shortest word in $S^{\ast}$ which represents $g$. When there is no risk of confusion we drop the subscript and write $|g|$ and $d(g,h)$ instead. For $r\in\NN$ we write $B(r)=\{g\in G:|g|\leq r\}$ for the ball of radius $r$. 
\subsection{Shift spaces}\label{subsec:shifts}

Let $A$ be a finite set and $G$ be a group. The \define{full $G$-shift} is the set $A^{G} = \{ x\colon G \to A\}$ equipped with the prodiscrete topology and with the (left) \define{shift} action $G \curvearrowright A^{G}$ by left multiplication given by 
\[ (g x)(h) = x(g^{-1}h) \qquad \mbox{  for every } g,h \in G 
\mbox{ and } x \in A^G. \]

In our setting, the set $A$ is called an \define{alphabet}, the elements $a\in A$ are called \define{symbols} and the elements $x \in A^G$ are called \define{configurations}. For $F\Subset G$, a \define{pattern} with support $F$ is an element $p \in A^F$. We denote the cylinder generated by $p$ by $[p] = \{ x \in A^{G} : x|_F = p \}$ and note that the cylinders are a clopen base for the prodiscrete topology on $A^{G}$, which is Hausdorff and compact. We say a pattern $p$ \define{appears} in $x$ (at position $g\in G$) if $g^{-1}x \in [p]$. 

\begin{definition}
	A set $X\subset A^G$ is called a $G$-\define{subshift} if it is  closed in the prodiscrete topology and invariant under the shift action.
\end{definition}

Equivalently, $X$ is a $G$-subshift if and only if there exists a set of forbidden patterns $\F$ such that \[X= X_{\F} = \{ x \in A^{G} : gx\notin [p] \mbox{ for every } g \in G, p \in \F  \}.\]

When the context is clear, we will drop the $G$ and plainly speak of a subshift. In what follows, we will say that a subshift is \define{non-trivial} if it is neither empty nor reduced to a single point.

Let $M,N$ be arbitrary subsets of $G$, and let $p\in A^M$, $q\in A^N$. We write $p\sqsubset q$ whenever $M\subset N$ and $q|_M=p$. If $p$ and $q$ are equal when restricted to $M\cap N$, then we denote by $p\vee q$ their \define{concatenation}, that is, the unique element of $A^{M\cup N}$ determined by the conditions $p\sqsubset (p\vee q)$ and $q\sqsubset (p\vee q)$.

Given $F \Subset G$, we denote by $\L_{F}(X)$ the set of all patterns $p \in A^{F}$ for which there is $x\in X$ such that $p$ appears in $x$. The \define{language} of $X$ is the set given by \[ \L(X) = \bigcup_{F \Subset G} \L_F(X).  \]
We note that two subshifts coincide if and only if their languages coincide.

Given a subshift $X\subset A^G$, we denote by $\operatorname{Fix}(X) = \{ g \in G : gx = x \mbox{ for every } x \in A^G\}$ the kernel of the shift action, that is, the normal subgroup of all elements in $G$ which fix every $x \in X$. Thus the shift action $G\curvearrowright X$ is faithful when $\fix(X)=\{\Id_G\}$. 

\subsection{The strong topological Markov property}

We begin with the classical notion of subshift of finite type.

\begin{definition}
    A subshift $X$ is of \define{finite type} (SFT) if there exists a finite set of forbidden patterns $\F$ for which $X = X_{\F}$.
\end{definition}

In dynamical terms, being an SFT is equivalent to shadowing, that is, the property that pseudo orbits of the system can be uniformly approximated by orbits, see for instance~\cite{Oprocha2008,Meyerovitch2017}. The following notion is a weaker symbolic version of shadowing which we need for the proof of~\Cref{thm:EmbeddingF2}.

\begin{definition}
    A $G$-subshift $X$ has the \define{strong topological Markov property} (strong TMP) if there exists $M\Subset G$ with the following property. For every $F\Subset G$ and $x,y\in X$ such that $x|_{FM\smallsetminus F} = y|_{FM\smallsetminus F}$, the element $x|_F\vee y|_{G\smallsetminus F}$ belongs to $X$. 

    A set $M$ as above is called a \define{memory constant} for $X$.
\end{definition}

The strong TMP was introduced in~\cite{BGMT_2020} based on earlier notions from~\cite{ChanHanGuanMarcusMey2014}, although it also appears earlier in the work of Gromov~\cite{Gromov1999} under the name ``splicable space''. This property has been generalized to the setting of continuous group actions on compact metrizable spaces~\cite{BarGarLi2022}. We remark that while for a countable group there are only countably many subshifts of finite type up to topological conjugacy, the class of subshifts with the strong TMP is in general uncountable~\cite{BGMT_2020}. We also note 
that every SFT satisfies the strong TMP.

\subsection{Strongly irreducible subshifts}

\begin{definition}
    A $G$-subshift $X$ is called \define{strongly irreducible} if there exists $K\Subset G$ such that if $S,T$ are two finite subsets of $G$ such that $SK \cap T = \varnothing$, then for every $p \in \L _{S}(X)$ and $q \in \L_{T}(X)$ there exists $x \in X$ such that $x|_{S} = p$ and $x|_{T} = q$.

    A set $K$ as above is called a \define{strong irreducibility constant} for $X$.
\end{definition}

Observe that a subshift $X\subset A^G$ is strongly irreducible with constant $K\Subset G$ when for every $S,T\Subset G$ with  $SK \cap T = \varnothing$, and for every $p \in \L _{S}(X)$ and $q \in \L _{T}(X)$, their concatenation $p \vee q$ belongs to $L_{S \cup T}(X)$.

If a subshift $X$ is strongly irreducible with a constant $K$ which does not contain $\Id_G$, then necessarily $X$ is trivial (because $\{\Id_G\}K \cap \{\Id_G\} = \varnothing$ and thus $|L_{\{\Id_G\}}(X)|=1$). It follows that a strong irreducibility constant of a non-trivial subshift must necessarily contain the identity. We also note that if $K$ is a strong irreducibility constant, so is any set which contains $K$, in particular $K \cup K^{-1}$. This is useful as we often want to take symmetric strong irreducibility constants.

\begin{definition}
    Let $G$ be a group, $K\Subset G$ and $(A_i)_{i \in I}$ a collection of finite subsets of $G$. We say $(A_i)_{i \in I}$ is $K$-disjoint if for every $i \in I$ we have \[ A_iK \cap \bigcup_{j \in I\smallsetminus\{i\}}A_j = \varnothing.  \]
\end{definition}

The next lemma is essential and will be used several times.

\begin{lemma}\label{lem:K-irreducible-mainlemma}
    Let $G$ be a group, $X$ a strong irreducible $G$-subshift with  constant $K$, and $(A_i)_{i \in I}$ a $K$-disjoint collection of subsets of $G$. For any collection of patterns $(p_i)_{i \in I}$ with $p_i \in L_{A_i}(X)$ there exists $x \in X$ such that $x|_{A_i}=p_i$ for all $i \in I$.
\end{lemma}
\begin{proof}
    Applying iteratively the property that $X$ is strongly irreducible shows that for all finite $J \subset I$ we have $\bigcap_{j \in J}([p_j]\cap X) \ne\varnothing$. It follows that the collection $\{[p_i]\cap X :i\in I\}$ of closed subsets of $X$ has the finite intersection property, and since $X$ is compact, it follows that  
    $\bigcap_{i\in I}([p_i]\cap X)\ne\varnothing$. Any element $x\in \bigcap_{i\in I}([p_i]\cap X)$ has the desired property.
\end{proof}

The following result states that non-trivial strongly irreducible subshifts are ``almost faithful'' in the sense that the kernel of the shift action is contained in the irreducibility constant and thus finite.
\begin{proposition}\label{lem:if-X-SI-then-Fix-contained-in-K}
Let $X$ be a non-trivial strongly irreducible
subshift with constant $K$. Then $\fix(X)\subset K$. 
\end{proposition}

\begin{proof} As $|X|>1$, there are two different symbols $a,b\in A$ that appear in $X$. Let $g\in G\smallsetminus K$, then we have $\{\Id_G\}K\cap\{g\}=\varnothing$. Then by strong irreducibility there is a configuration $x\in X$ with $x(\Id_G)=a$ and $x(g)=b$ thus $x\ne gx$. This shows that $g\not\in\fix(X)$. \end{proof}

\begin{remark}
    Let $G$ be a group whose only normal finite subgroup is $\{\Id_G\}$, for instance, a torsion-free group such as $\ZZ^d$. The previous lemma shows that the shift action is faithful on every non-trivial strongly irreducible $G$-subshift.
\end{remark}

\begin{example}\label{fix-can-be-nontrivial-if-torsion-elements}
    There exist non-trivial strongly irreducible subshifts where the shift action is non-faithful. For instance if $F$ is a finite group and $G=\Z\times F$, we can consider \[
    X=\{x\in\{0,1\}^G: \mbox{ for every } t\in F, \ (0,t)x=x\}.
    \]
    Then $X$ is a strongly irreducible $G$-subshift with constant $K = \{0\}\times F$ and $\fix(X)=K$.
\end{example}

Next we will show that the language of a non-trivial strongly irreducible subshift is rich, more precisely, we will show that it grows exponentially with respect to the size of its support. For that we will need the following elementary lemma.

\begin{lemma}
\label{lem:conjunto-separado}
Let $G$ be a group, and let $K,F\Subset G$
be nonempty sets with $K$ symmetric. There exists $U\subset F$
with $|U||K| \geq |F|$ and such that $(\{g\})_{g \in U}$ is $K$-separated.\end{lemma}

\begin{proof}
    Let $U\subset F$ be a maximal set with the property that for every $g \in U$ we have $gK \cap (U\smallsetminus \{g\}) = \varnothing$. As $K$ is symmetric, we have that $F\subset UK$, otherwise we could take $h \in F\smallsetminus UK$ and $U'= U \cup \{h\}$ would also satisfy the property. We obtain $|U||K|\geq |UK|\geq |F|$.
\end{proof}



\begin{proposition}
\label{prop:lenguaje-de-un-subshit-SI-tiene-muchos-elementos} Let
$G$ be a group, let $X$ be a non-empty strongly irreducible $G$-subshift, and let $K\Subset G$ be a strong irreducibility
constant for $X$. For every $F\Subset G$ we have \[|L_{F}(X)|\geq \left\vert L_{\{\Id_G\}}(X) \right\vert^{\frac{|F|}{2|K|}}.\]
\end{proposition}
\begin{proof}
Let $K'=K\cup K^{-1}$. By~\Cref{lem:conjunto-separado} there exists $U\subset F$ with $|K'||U|\geq |F|$ and such that $(\{g\})_{g \in U}$ is $K'$-separated. As $X$ is strongly irreducible,~\Cref{lem:K-irreducible-mainlemma} implies that for each map $\sigma \colon U \to L_{\{\Id_G\}}(X)$ there exists $x_{\sigma}\in X$ such that $x_{\sigma}(g)=\sigma(g)(\Id_G)$ for every $g \in U$. We conclude that \[ \left\vert L_F(X)\right\vert \geq \left\vert L_U(X)\right\vert \geq \left\vert L_{\{\Id_G\}}(X)\right\vert^{|U|} \geq \left\vert L_{\{\Id_G\}}(X)\right\vert^{\frac{|F|}{|K'|}} \geq \left\vert L_{\{\Id_G\}}(X)\right\vert^{\frac{|F|}{2|K|}}. \qedhere  \]
\end{proof}

In particular, if $X$ is a non-trivial strongly irreducible $G$-subshift, we get that for every $F\Subset G$ we have \[|L_{F}(X)|\geq 2^{\frac{|F|}{2|K|}}.\]

\subsection{Automorphism groups and elementary results}\label{subsection:automorphism-groups-basic-things}

Given two $G$-subshifts $X$ and $Y$, a map $\phi\colon X\to Y$ is called a \define{morphism} if it is continuous and $G$-equivariant (that is, if $g\phi(x)=\phi(gx)$ for every $x \in X$ and $g \in G$). The space of all self morphisms of a subshift $X$ is denoted by $\operatorname{End}(X)$. The Curtis-Hedlund-Lyndon Theorem provides a characterization of these morphisms through local maps, see~\cite[Theorem 1.8.1]{ceccherini-SilbersteinC09} for a modern proof. This result will be used many times in what follows, we shall refer to it as the \define{CHL theorem}. 

\begin{theorem}[CHL theorem]\label{thm:curtis_lindon_hedlund}
   A map $\phi\colon A^G \to B^G$ is a morphism if and only if there is $F\Subset G$ and $\Phi \colon A^F\to B$ such that $\phi(x)(g) = \Phi((g^{-1} x)|_{F})$ for every $x \in A^G$, $g \in G$.
\end{theorem}

A morphism $\phi \colon X \to Y$ is called a \define{topological factor map} if it is surjective and a \define{topological conjugacy} if it is bijective.

\begin{definition}
    The automorphism group of a $G$-subshift $X$ is the set $\Aut(X)$ of all $G$-equivariant homeomorphisms $\varphi \colon X \to X$ with composition as the group operation.
\end{definition}

In other words, $\Aut(X)$ is the group of all topological self-conjugacies of $X$. A natural consequence of the CHL theorem is that for every countable group $G$ and $G$-subshift $X$, we have that $\Aut(X)$ is countable. 

\begin{definition}
    The \define{right shift} by $g\in G$ is the map $\tau_g \colon A^G \to A^G$ given by \[ \tau_g(x)(h) = x(hg) \mbox{ for all } x \in A^G, h \in G.   \]
\end{definition}

We note that for any $g \in G$ the right shift map $\tau_g$ is $G$-equivariant (with respect to the left shift action). Indeed, for every $h \in G$ we have \[ \tau_g(hx)(t) = (hx)(tg) = x(h^{-1}tg) = \tau_g(x)(h^{-1}t) = h(\tau_g(x))(t).   \]  
However, the analogous left shift map $x \mapsto g^{-1}x$ is not $G$-equivariant unless $g\in Z(G)$, in which case it coincides with the right shift.

\begin{proposition}\label{prop:G_in_aut}
    For every group $G$, we have that $G$ embeds into $\aut(A^{G})$.
\end{proposition}

\begin{proof}
    Let $g \in G$. The right shift map $\tau_g$ is a self homeomorphism of $A^G$ which is $G$-equivariant, thus $\tau_g \in \Aut(A^G)$. Clearly $\tau_g \circ \tau_h = \tau_{gh}$ and $\tau_g = \tau_h$ if and only if $g=h$, therefore $g \mapsto \tau_g$ is a monomorphism from $G$ to $\Aut(A^G)$.
\end{proof}

    Let $G$ be a group and $X$ be a $G$-subshift. In general it is not true that the right shifts are self homeomorphisms of $X$, since $X$ is not necessarily invariant by right shifts. In fact, it is possible to construct infinite minimal SFTs in the free group on two generators with trivial automorphism group.

    \begin{example}\label{ex:trivial_aut_group}
    Let $F_2$ be the free group on generators $a,b$ and let $A = \{a,a^{-1},b,b^{-1}\}$. Consider the subshift \[X = \{x \in A^{F_2} : \mbox{ for every } g \in F_2 \mbox{ if } x(g) = t, \mbox{ then for } s \in A\smallsetminus \{t\}, x(gs) = s^{-1}\}.\]
    One can verify that $X$ is an uncountable and minimal $F_2$-SFT. In fact, the shift action $F_2 \curvearrowright X$ is a symbolic representation of the natural action of $F_2$ on the metric boundary of its Cayley graph. 
    
    We argue that $\Aut(X)$ is trivial. For $s \in A$, let $C(s)$ denote the cone of all reduced words in $F_2$ which begin by $s$. It can be verified that if $x \in X$ is such that $x(\Id_{F_2})=t$, then for every $s \in A \smallsetminus \{t\}$ the values of $x$ are fixed on $C(s)$ and must ``point towards $\Id_{F_2}$'', more precisely, if $w = sva \in C(s)$ is a reduced word with $v \in S^*$ and $a \in A$, then $x(sva) = a^{-1}$. We illustrate this property in~\Cref{fig:grupolibre}.
    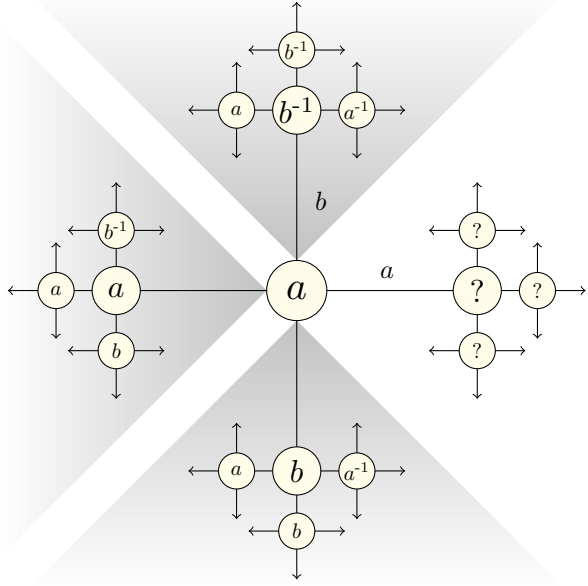
\begin{figure}[ht!]
        \centering
        \begin{tikzpicture}[scale=0.8]
			\def \c{0.5}
			\def \b{0.4}
			\def \a{0.3}

            \fill[opacity = 0.6, gray!80, path fading=west] (-0.5,0) -- (-5,4.5) --(-5,-4.5) -- cycle;
            \fill[opacity = 0.6, gray!80, path fading=north] (0,0.5) -- (-4.5,5) --(4.5,5) -- cycle;
           \fill[opacity = 0.6, gray!80, path fading=south] (0,-0.5) -- (-4.5,-5) --(4.5,-5) -- cycle;
			
			\draw [->] (0,0) to (-4.8,0);
			\draw [->] (0,0) to (4.8,0);
			\draw [->] (0,0) to (0,4.8);
			\draw [->] (0,0) to (0,-4.8);
			
			\node at (1.5,0.3) {$a$};
			\node at (0.4,1.5) {$b$};
			
			\draw [->] (-3,0) to (-3,-1.8);
			\draw [->] (-3,0) to (-3, 1.8);
			
			\draw [->] (3,0) to (3,-1.8);
			\draw [->] (3,0) to (3, 1.8);
			
			\draw [->] (0,3) to (-1.8,3);
			\draw [->] (0,3) to (1.8, 3);
			
			\draw [->] (0,-3) to (-1.8,-3);
			\draw [->] (0,-3) to (1.8, -3);

			\draw[fill = yellow!10] (0,0) circle (\c);
            \node at (0,0) {\scalebox{1.5}{$a$}};
			\draw[fill = yellow!10] (-3,0) circle (\b);
            \node at (-3,0) {\scalebox{1.2}{$a$}};
			\draw[fill = yellow!10] (3,0) circle (\b);
            \node at (3,0) {\scalebox{1.2}{$?$}};
			\draw[fill = yellow!10] (0,3) circle (\b);
            \node at (0,3) {\scalebox{1.2}{$b^{\text{-}1}$}};
			\draw[fill = yellow!10] (0,-3) circle (\b);
            \node at (0,-3) {\scalebox{1.2}{$b$}};

			\draw [->] (-4,0) to (-4,0.8);
			\draw [->] (-4,0) to (-4,-0.8);
			\draw[fill = yellow!10] (-4,0) circle (\a);
            \node at (-4,0) {\scalebox{0.8}{$a$}};
			\draw [->] (-3,1) to (-3.8,1);
			\draw [->] (-3,1) to (-2.2,1);
			\draw[fill = yellow!10] (-3,1) circle (\a);
            \node at (-3,1) {\scalebox{0.8}{$b^{\text{-}1}$}};
			\draw [->] (-3,-1) to (-3.8,-1);
			\draw [->] (-3,-1) to (-2.2,-1);
			\draw[fill = yellow!10] (-3,-1) circle (\a);
            \node at (-3,-1) {\scalebox{0.8}{$b$}};

			\draw [->] (4,0) to (4,0.8);
			\draw [->] (4,0) to (4,-0.8);
			\draw[fill = yellow!10] (4,0) circle (\a);
            \node at (4,0) {\scalebox{0.8}{$?$}};
			\draw [->] (3,1) to (3.8,1);
			\draw [->] (3,1) to (2.2,1);
			\draw[fill = yellow!10] (3,1) circle (\a);
            \node at (3,1) {\scalebox{0.8}{$?$}};
			\draw [->] (3,-1) to (3.8,-1);
			\draw [->] (3,-1) to (2.2,-1);
			\draw[fill = yellow!10] (3,-1) circle (\a);
            \node at (3,-1) {\scalebox{0.8}{$?$}};
			
			\draw [->] (0,4) to (-0.8,4);
			\draw [->] (0,4) to (0.8,4);
			\draw[fill = yellow!10] (0,4) circle (\a);
            \node at (0,4) {\scalebox{0.8}{$b^{\text{-}1}$}};
			\draw [->] (-1,3) to (-1,3.8);
			\draw [->] (-1,3) to (-1,2.2);
			\draw[fill = yellow!10] (-1,3) circle (\a);
            \node at (-1,3) {\scalebox{0.8}{$a$}};
			\draw [->] (1,3) to (1,3.8);
			\draw [->] (1,3) to (1,2.2);
			\draw[fill = yellow!10] (1,3) circle (\a);
            \node at (1,3) {\scalebox{0.8}{$a^{\text{-}1}$}};
			
			\draw [->] (0,-4) to (-0.8,-4);
			\draw [->] (0,-4) to (0.8,-4);
			\draw[fill = yellow!10] (0,-4) circle (\a);
            \node at (0,-4) {\scalebox{0.8}{$b$}};
			\draw [->] (-1,-3) to (-1,-3.8);
			\draw [->] (-1,-3) to (-1,-2.2);
			\draw[fill = yellow!10] (-1,-3) circle (\a);
            \node at (-1,-3) {\scalebox{0.8}{$a$}};
			\draw [->] (1,-3) to (1,-3.8);
			\draw [->] (1,-3) to (1,-2.2);
			\draw[fill = yellow!10] (1,-3) circle (\a);
            \node at (1,-3) {\scalebox{0.8}{$a^{\text{-}1}$}};
			\node (1) at(1,1) {};
		
			\end{tikzpicture}
        
        \caption{Three cones are fixed by the symbol at the origin.}
        \label{fig:grupolibre}
    \end{figure}
    
    Given an automorphism $\varphi \in \Aut(X)$ and $x \in X$, then $x$ and $\varphi(x)$ must have the same values on at least two cones (the ones given by the generators in $A \smallsetminus \{ x(\Id_{F_2}), \varphi(x)(\Id_{F_2}) \}$). The CHL theorem implies that $\varphi$ must locally act as the identity on any large enough pattern contained in these cones, and minimality implies that every pattern occurs in those cones, thus $\varphi$ must be the identity. 
\end{example}

The previous issue does not arise if we consider the right shift by elements in $Z(G)$. In this case the map coincides with the left shift and $G$-invariance of $X$ ensures that $\tau_g$ is a homeomorphism. Modding out by $\operatorname{Fix}(X)$ yields the following general result.

\begin{proposition}\label{center-of-G-embeds-in-Aut(G)}
Let $X\subset A^{G}$ be a subshift. Then $\langle\tau_g : g\fix(X)\in Z(G/\fix(X))\rangle$ is a subgroup of $Z(\Aut(X))$, and is isomorphic to $Z\left(\quot{G}{\operatorname{Fix}(X)}\right)$.   
\end{proposition}
\begin{proof}
We check that $\tau_g\in Z(\Aut(X))$ for every $g\in G$ such that $g\fix(X) \in Z(G/\fix(X))$, and thus $\langle\tau_g : g\fix(X)\in Z(G/\fix(X))\rangle \leqslant Z(\Aut(X))$.  Let $g \in G$ as above. Since $X$ is $G$-invariant, it is also $G/\fix(X)$-invariant under the natural action. Since $g\fix(X) \in Z(G/\fix(X))$, it follows that $\tau_g(x) = g^{-1}x$ for every $x \in X$ and thus that $\tau_g \in \Aut(X)$. Moreover, if $\varphi \in \Aut(X)$ we have that \[ \varphi(\tau_g(x)) =  \varphi(g^{-1}x) = g^{-1}\varphi(x) = \tau_g(\varphi(x)).  \]
Thus $\tau_g \in Z(\Aut(X))$.

The verification that $\langle\tau_g : g\fix(X)\in Z(G/\fix(X))\rangle$ is isomorphic to $Z(G/\fix(X))$ is elementary, but we provide the details for the interested reader. Define a group homomorphism $f\colon Z\left(G/\operatorname{Fix}(X)\right)\to\langle\tau_g : g\fix(X)\in Z(G/\fix(X))\rangle$ by setting $f(g\fix(X))=\tau_g$. The map $f$ is well defined since $\tau_g=\tau_{g'}$ when  $g'g^{-1}\in\fix(X)$, and it is clear that $f$ is an injective group homomorphism. The definition of $f$ shows that its range is equal to $\{\tau_g : g\fix(X)\in Z(G\fix(X))\} = \langle\tau_g : g\fix(X)\in Z(G/\fix(X))\rangle$. 
\end{proof}

 The following theorem will not be needed in the rest of our article, but we find it worthwhile to include, as it is the natural generalization of~\cite[Theorem 3.1]{BoyleLindRudolph} to arbitrary groups.

\begin{theorem}
    Let $G$ be a group and $X\subset A^G$ a subshift. If the set of points with finite orbit of $G\curvearrowright X$ is dense in $X$, then $\Aut(X)$ is residually finite. In particular, $\Aut(A^{G})$ is residually finite if and only if $G$ is residually finite.
\end{theorem}
    \begin{proof}
   Let $H\leqslant G$ and let $\stab_{H}(X)=\{x\in X\colon hx=x\text{ for all } h\in H\}$. We remark that when $H$ has finite index this set is finite. Moreover, for every $\varphi\in\aut(X)$ the set $\stab_{H}(X)$ is invariant under $\varphi$, thus the restricted map $\varphi|_{\stab_{H}(X)} \colon \stab_{H}(X) \to \stab_{H}(X)$ induces a permutation in $\operatorname{Sym}(\stab_{H}(X))$. 
   
   Let $K_H$ be the kernel of the restriction map $\psi_H \colon \Aut(X) \to \operatorname{Sym}(\stab_{H}(X))$ given by $\psi_H(\varphi)=\varphi|_{\stab_{H}(X)}$. Notice that $K_H$ is a normal finite index subgroup of $\Aut(X)$ whenever $H$ is a finite index subgroup of $G$. Suppose that the set of points with finite orbit is dense in $X$, we claim that \[\bigcap_{\substack{H\leqslant G \\ [G:H]<\infty }}K_{H}=\{\Id_{\aut(X)}\}.\] 
   
   Let $\varphi\in\bigcap_{\substack{H\leqslant G \\ [G:H]<\infty }}K_{H}$, and let $\Phi\colon A^F\to A$ be a local rule defining $\varphi$ as in the CHL theorem. Without loss of generality we assume $1_G\in F$. Let $p\in L_F(X)$. Since $[p]$ is open and configurations with finite orbit are dense in $X$, there is $x\in X$ with finite orbit and $x\in [p]$. According to the definition of $\varphi$ we have $\varphi(x)(\Id_G)=\Phi(x|_F)=\Phi (p)$. Since $\varphi(x)=x$, it follows that $\Phi(p)=x(\Id_G)=p(\Id_G)$. This proves that the local rule $\Phi$ satisfies $\Phi(p)=p(\Id_G)$ for every $p\in L_F(X)$, and therefore $\varphi$ is the identity map on $X$.  We conclude that $\varphi = \Id_{\Aut(X)}$ and thus that $\aut(X)$ is a residually finite group. 
   
   In the case when $X=A^G$ it is well-known that the set of points with finite orbits is dense when $G$ is residually finite, see~\cite[Theorem 2.7.1]{ceccherini-SilbersteinC09} and thus $\Aut(A^G)$ is residually finite.
  The converse follows from~\Cref{prop:G_in_aut} and the fact that every subgroup of a residually finite group is residually finite.\end{proof}

We finish this section collecting some basic facts about restrictions of strongly irreducible subshifts. 
\begin{proposition}\label{prop:basic-things-SI}
    Let $G$ be a group and let $X\subset A^G$ be a non-trivial strongly irreducible $G$-subshift with constant $K$. Let $H\leqslant G$ with $K\subset H$, and consider the $H$-subshift $X|_H=\{x|_H  \in A^H : x\in X\}$. Then:
\begin{enumerate}[(1)]
\item $K$ is a strong irreducibility constant for $X|_H$.
\item $\fix(X|_H)=\fix(X)$.
\item $X$ can be defined by a set of forbidden patterns whose support is contained in $H$. 
\item For every $x\in X$ and $y\in X|_H$, $x|_{G\smallsetminus H}\lor y$ lies in $X$. 
\item $X$ has the strong TMP if and only if $X|_H$ has the strong TMP.
\item $\Aut(X|_H)$ embeds into $\Aut(X)$. 
\end{enumerate}
\end{proposition}
\begin{proof}
    (1) follows directly from the definition. Regarding (2), it is clear that $\fix(X)\subset \fix(X|_H)$. Let $h\notin \fix(X)$, then $hx(g)\ne x(g)$ for some $x\in X$ and $g \in G$, and thus if we let $y=g^{-1}x$, then $(hy)(\Id_G) \neq y(\Id_G)$ and so $(hy)|_H\ne y|_H$. Therefore $h\notin\fix(X|_H)$.

    In order to prove (3), take $\F = L(A^H) \smallsetminus L(X|_H)$ and let $X_{\mathcal F}\subset A^G$ be the $G$-subshift defined by forbidding $\mathcal F$. We will prove that $X=X_{\mathcal F}$. The inclusion $X\subset X_\mathcal F$ is clear. For the other inclusion it suffices to prove $L(X_{\mathcal{F}})\subset L(X)$. Let $p\in L(X_{\mathcal{F}})$ be a pattern with support $F\Subset G$. We partition $F=g_1F_1\sqcup\dots\sqcup g_nF_n$ such that for distinct $i,j \in \{1,\dots,n\}$, $F_i\Subset H$ and $g_iH \cap g_jH = \varnothing$.
    
    For each $F_i$, we write $p_i=p|_{g_iF_i}$. It is clear from the definition of $X_\mathcal{F}$ that every $p_i$ must belong to $L(X)$. Since $K$ is a strong irreducibility constant for $X$ and $(F_i)_{i=1}^n$ is $K$-disjoint, it follows that $p=p_1\vee\dots\vee p_n$ belongs to $L(X)$.  

    For the proof of (4) we use the same set $\mathcal F$ from (3). If $x\in X$ and $y\in X|_H$, then $x|_{G\smallsetminus H}\lor y$ contains no forbidden pattern from $\mathcal F$, and thus it lies in $X=X_{\mathcal F}$. 

    Next we prove (5). Suppose first that $X|_H$ has the strong TMP with memory constant $M\Subset H$. We claim that $M$ is also a memory constant for $X$. Indeed, let $F\Subset G$ and take $x,y\in X$ such that $x|_{FM\smallsetminus F}=y|_{FM\smallsetminus F}$. As in the proof of (3), partition $F=g_1F_1\sqcup\dots\sqcup g_nF_n$ into disjoint left $H$-cosets and for $i \in \{1,\dots,n\}$ take $x_i= (g_i^{-1}x)|_H$ and $y_i= (g_i^{-1}y)|_H$. It follows that $x_i|_{F_iM \smallsetminus F_i} = y_i|_{F_iM \smallsetminus F_i}$ and thus $z_i = x_i|_{F_i} \vee y_i|_{H\smallsetminus F_i} \in X|_H$.

    Take $w_0 =y$ and for $i \in \{1,\dots,n\}$ let $w_i = g_i((g_i^{-1}w_i)|_{G\smallsetminus H} \vee z_i)$. A simple computation shows that $w_n = x|_{F} \vee y_{G\smallsetminus F}$. Iterating the result of part (4), we have that each $w_i \in X$ and thus $X$ has the strong TMP with constant $M$.

    Conversely, suppose that $X$ has the strong TMP with constant $M$ and consider $M'= M \cap H$. Indeed, let $F\Subset H$ and take $x,y \in X|_H$ such that $x|_{FM'\smallsetminus F}=y|_{FM'\smallsetminus F}$. Take any $w \in X$, by part (4) we have $x' = w|_{G\smallsetminus H} \vee x \in X$ and $y' = w|_{G\smallsetminus H} \vee y \in X$. It follows that $x'|_{FM\smallsetminus F}=y'|_{FM\smallsetminus F}$ and thus $z = x'|_{F}\vee y'|_{G\smallsetminus F} \in X$. As $z|_{H} = x|_{F}\vee y|_{H\smallsetminus F} \in X$, we conclude that $X_{H}$ has the strong TMP with constant $M'$. 

    Finally we show (6): we define an injective homomorphism $f\colon \Aut(X|_H)\to \Aut(X)$ as follows. Let $\phi\in \Aut(X|_H)$, and let $\Phi\colon A^F\to A$, $F\Subset H$ be a local rule that defines $\phi$, which exists by the CHL theorem. We denote by $f(\phi)$ the morphism $ f(\phi)\colon X\to X$ defined by the same local rule, that is $ \bigl(f(\phi)(x)\bigr)(g)=\Phi((g^{-1}x)|_F)$ for $x\in X$ and $g\in G$. 
    
    A first observation is that the association $\phi\to f(\phi)$ is independent of the chosen local rule. Furthermore, item (3) shows that this map is well defined in the sense that $ f(\phi)(x)$ belongs to $X$ for every $x\in X$. Furthermore, a direct computation shows that $f(\phi\circ\psi)=f(\phi)\circ f(\psi)$ for every $\phi,\psi\in\Aut(X|_H)$. Let $\phi\in\Aut(X)$, let $\phi^{-1}$ be its inverse in $\Aut(X|_H)$, and observe that $f(\phi)\circ f(\phi^{-1})=f(\phi^{-1})\circ f(\phi)$ is the identity map on $X$. It follows that  $f(\phi)$ is invertible and its inverse $f(\phi^{-1})$ is continuous. Thus  $f(\phi)$ belongs to $\Aut(X|_H)$ for every $\phi\in\Aut(X)$. Finally, it is obvious that if $f(\phi)$ acts trivially on $X$ then $\phi$ must act trivially on $X|_H$, therefore $f$ is injective.\end{proof}
    

\section{Markers on infinite groups}\label{sec:markers}


It will be convenient for us to consider translates of patterns. Given a pattern $p\in A^F$ and $g\in G$, we denote by $gp$ the pattern with support $gF$ defined by $gp(gh)=p(h)$ for every $h\in F$ (or equivalently, $gp(t)=p(g^{-1}t)$ for every $t\in gF$).  We say that $p$ is $g$\define{-overlapping} if $p(h)=gp(h)$ for every $h\in F\cap gF$. Informally, this corresponds to the intuition that the patterns $p$ and $gp$ can be ``overlapped''. If the previous condition fails, we say that $p$ is \define{non-$g$-overlapping}. 

Furthermore, it will be useful in this section to use the following notation: for a group $G$ endowed with a word metric, write $\diam(G) = \sup_{g \in G}|g|$. Also, for $0 \leq r<R$, we will write $B(r,R)= \{ g \in G: r< |g|\leq R\} = B(R)\smallsetminus B(r)$ for the ring of inner radius $r+1$ and outer radius $R$.

\begin{definition}
\label{def:markers}\label{def:subshift-with-abundant-markers}Let $G$ be a group and let $X$ be a $G$-subshift. Given $Y\subsetneq W\Subset G$, a $(Y,W)$-\define{marker} for $X$ is a pattern $p\in L_{W\smallsetminus Y}(X)$ that is non-$g$-overlapping for all $g\in WY^{-1}\smallsetminus \fix (X)$.  We say that $X$ is \define{markable} when for every $Y\Subset G$  there exists $W\Subset G$ and a $(Y,W)$-marker for $X$.
\end{definition}

\begin{figure}[ht!]
    \centering
    \begin{tikzpicture}[scale = 0.5]
    \yolk
    \end{tikzpicture}
    \caption{A $( \{2,3\}^2, \{0,\dots,5\}^2)$-marker in $\{0,1\}^{\ZZ^2}$.}
    \label{fig:marcador}
\end{figure}

Previous works have established sufficient conditions for SFTs on $\ZZ^d$ to be markable. For example, Ward \cite{Ward1994} proved that being infinite and strongly irreducible is a sufficient condition, and Hochman \cite{Hochman2010} proved that positive topological entropy is also a sufficient condition.  The main result of this section is that without the hypothesis of being an SFT, strong irreducibility is a sufficient condition for being markable, and that this is valid on every infinite group.

\begin{theorem}[\Cref{thm:markers}]
Every nonempty strongly irreducible subshift on an infinite group is markable. Furthermore, if the group is finitely generated and endowed with a word metric, then the subshift admits $(B(r),B(37r))$-markers for all large enough $r$. 
\end{theorem}

\begin{remark} 
    Our definition of marker is slightly different than the one used in \cite{Ward1994,Hochman2010}. They coincide (in $\ZZ^d$) for a subshift $X$ where the action is faithful (that is, $\fix(X)=\{\Id_G\}$). Otherwise our definition is more flexible in the sense that a larger class of subshifts is allowed to have markers. For example, the subshift \[X = \{x\in\{0,1\}^{\Z^2}:\forall a,b\in\Z \  x(a,b)=x(a,b+1)\},\] is markable according to our definition, while no pattern in its language is a marker according to the definition in \cite{Ward1994,Hochman2010}. 
\end{remark}

\begin{remark}
	A trivial subshift is strongly irreducible and markable for the trivial reason that $\operatorname{Fix}(X)=G$. With the exception of this case, a strongly irreducible subshift $X$ will always satisfy $|\operatorname{Fix}(X)|<\infty$ by \Cref{lem:if-X-SI-then-Fix-contained-in-K}, and it will be important for our applications that \Cref{thm:markers} covers these cases.
\end{remark}

Below are some useful properties of markers. The following monotonicity property of markers follows directly from the definition.
\begin{observation}\label{obs:markers-can-be-completed}
	If $p$ is a 
	$(Y,W)$-marker for a subshift $X$ and $Y'\subset Y$ and $q\in L_{W\smallsetminus Y'}(X)$ are such that $q|_{W\smallsetminus Y}=p$, then $q$ is a $(Y',W)$-marker for $X$. 
\end{observation}

The following are some reformulations of the property of being markable.

\begin{proposition}\label{prop:alternative-notions-for-markers}
	The following properties are equivalent for a subshift $X$ on a countably infinite group $G$.
	\begin{enumerate}
		\item  For every $Y\Subset G$ there is a sequence of $(Y,W_n)$-markers for $X$ such that $\bigcup_{n\in\N} W_n=G$.
		\item $X$ is markable.
		\item There exists $Y_0 \Subset G$ such that for every $Y\Subset G$  with $Y_0 \subset Y$, there exists $(Y,W)$-marker in $X$.
		\item There is a sequence of $(Y_n,W_n)$-markers for $X$ such that $Y_n\subset Y_{n+1}$ for all $n$, and $\bigcup_{n\in\N} Y_n=G$.
	\end{enumerate}
\end{proposition}
\begin{proof}
    	The implications $1\Rightarrow 2 \Rightarrow 3\Rightarrow 4$ are direct, and $4\Rightarrow 1$ follows from \Cref{obs:markers-can-be-completed}.         
\end{proof}

The fourth condition shows that in the case of a finitely generated group with a word metric, if a subshift $X$ admits a constant $\lambda> 1$ such that $X$ admits $(B(r),B(\lambda r))$-markers for all $r$ large enough, then $X$ is markable. We do not expect the constant $\lambda = 37$ in \Cref{thm:markers} to be optimal, but lowering this constant has no effect in our applications.

\subsection{Proof outline}

The core of the proof of~\Cref{thm:markers} is contained in the following technical result, which applies both to finite and infinite finitely generated groups:
\begin{proposition}
\label{prop:non-overlapping-multiuso}Let $G$ be a finitely generated
group endowed with a word metric, and let $X$ be a nontrivial strongly irreducible $G$-subshift with a symmetric constant $K$. Let $k\geq 1$ such that $K\subset B(k)$ and let $r$ be a positive integer which verifies the following conditions:

\begin{enumerate}
\item $|B(38r)|< |L_{B(r-k)}(X)|$. 
\item $\diam(G) > 38r$.
\item $r>16k|K^3|+2k$. 
\end{enumerate}
Then $X$ admits a $(B(r),B(37r))$-marker. 
\end{proposition}

Let us stress that the conditions which involve balls in this result are linked to the underlying word metric. We shall prove later that if $G$ is infinite, then every large enough $r$ satisfies the conditions in \Cref{prop:non-overlapping-multiuso}. This will prove~\Cref{thm:markers} for finitely generated groups. In the case of locally finite groups we will prove~\Cref{thm:markers} by applying~\Cref{prop:non-overlapping-multiuso} to a well-chosen sequence of finite subgroups. The proof of \Cref{thm:markers} for a general infinite group will be reduced to these two cases using~\Cref{prop:basic-things-SI}.

The strategy for proving \Cref{prop:non-overlapping-multiuso} will be as follows. Fix $G$, $X$, $K$, $k$ and $r$ as in the statement. We will construct separately three patterns, that we call \define{long range}, \define{mid range}, and \define{short range}.

\begin{enumerate}
\item \textbf{Long range}: a pattern $p_l\in L(X)$ with support $B(6r,37r)$ that is non-$g$-overlapping for all $g\in B(2r,38r)$.

\item \textbf{Mid range}: a pattern $p_m\in L(X)$ with support $B(r,5r)$ that is non-$g$-overlapping for $g\in B(2r)\smallsetminus K^3$.

\item \textbf{Short range}: a pattern $p_s\in L(X)$ with support
$B(8k|K^3|)$ that is non-$g$-overlapping for all $g\in K^3\smallsetminus\fix(X)$. 
\end{enumerate}

The strong irreducibility of $X$ implies the existence of a pattern $p\in L_{B(r,37r)}(X)$ having $p_l$, $p_m$, and a translation of $p_s$ as subpatterns. The properties of $p_l$, $p_m$, and $p_s$ ensure that $p$ is a $(B(r),B(37r))$-marker. 

Our construction of long range patterns is inspired in the arguments from~\cite[Section 3.2]{Hochman2010}. However, these ideas can not be directly applied if $G$ has torsion elements, and this problem is solved by combining mid range and long range patterns. Despite the fact that this is rather hidden in the arguments, this is the main difficulty addressed in our proof. 

\subsection{Proof of \Cref{prop:non-overlapping-multiuso}} 
Let $G$ be a group, let $F\Subset G$, and let $p\in A^F$ be a pattern. We remind the reader that $gp$ is the pattern with support $gF$ defined by $gp(h)=p(g^{-1}h)$, that $p$ is $g${-overlapping} when $p(h)=gp(h)$ for every $h\in F\cap gF$, and non-$g$-overlapping if the latter condition fails. The following facts follow directly from the definitions:  
\begin{observation}\label{monotony-of-overlappings}
    Let $p\sqsubset q$ be patterns. If $q$ is $g$-overlapping, then the same is true for $p$. Thus, if $p$ is non-$g$-overlapping, the same is true for $q$.
\end{observation}
\begin{observation}\label{patterns-with-verlappings}
    The pattern $p\in A^F$ is $g$-overlapping if and only if $p(h)=p(gh)$ for every $h\in F\cap g^{-1}F$. This follows from the fact that $p$ is $g$-overlapping if and only if $gp$ is $g^{-1}$-overlapping.
\end{observation}
For the construction of long range patterns we will need the following elementary result. 
\begin{lemma}
\label{lem:most-very-ugliest-lemma}  Let Let $r \geq 4$ and $G$ be a finitely generated group endowed with a word metric such that $\diam(G)>38r$. For all $g\in B(38r)$
there is $h\in G$ with $hB(r)\subset B(6r,37r)\cap gB(6r,37r)$. 
\end{lemma}
\begin{proof}
Suppose first
that $14r+2\leq|g|\leq 38r$. As $\diam(G)>38r$, there exists a sequence of group elements
$\Id_{G}=g_{0},g_{1},\dots,g_{n}=g$ with $n=|g|$, and $d(g_{i},g_{i+1})=1$
for all $i=0,\dots,n-1$. If $n$ is even we let $h=g_{n/2}$, and
otherwise $h=g_{(n+1)/2}$. In both cases we have that $d(\Id_G,h)\geq 7r+1$
and $d(g,h)\geq7r+1$. It easily follows from triangle inequalities
that $hB(r)\subset B(6r,37r)\cap gB(6r,37r)$. 

Suppose now $|g|\leq14r+1$, and let $h$ be any group element with $|h|=21r+2$. It is clear that with this choice we have $hB(r)\cap B(6r)=\varnothing$, $hB(r)\cap gB(6r)=\varnothing$, and $hB(r)\subset B(37r)$. We also have $d(h,g)\leq |g|+|h|\leq 35r+4$, which implies that all elements in $hB(r)$ are at distance at most $36r+4$ from $g$. Since $r\geq 4$, it follows that $hB(r)\subset gB(37r)$. Thus $hB(r)$ has the desired property $hB(r)\subset B(6r,37r)\cap gB(6r,37r)$.
\end{proof}
 We now proceed to construct the aforementioned long range, mid range, and short range patterns.

\begin{lemma}[Long range]
\label{lem:long-range-markers} Let $G$, $X$, $K$, $k$ and $r$ be as in \Cref{prop:non-overlapping-multiuso}. Then there is a pattern $p\in L_{B(6r,37r)}(X)$ that is non-$g$-overlapping for all $g\in B(2r,38r)$. \end{lemma}

\begin{proof}
For each $g\in B(2r,38r)$ we denote by  $M_{r}(g)$ be the set of patterns in $L_{B(6r,37r)}(X)$ that are $g$-overlapping. 

We now prove an upper bound for the cardinality of $M_r(g)$, for fixed $g\in B(2r,38r)$. This is done by examining the intersection $B(6r,37r)\cap gB(6r,37r)$. By \Cref{lem:most-very-ugliest-lemma}, there is an element $h\in G$ such that $hB(r)\subset B(6r,37r)\cap gB(6r,37r)$. 

Let us observe that for every  $f\in hB(r)$, the element $g^{-1}f$ is in $B(6r,37r)$, but not in $hB(r)$. Indeed, $g^{-1}f\in B(6r,37r)$ follows from the fact that  $hB(r)\subset gB(6r,37r)$, and $g^{-1}f\not\in hB(r)$ holds because $|g|>2r$. 

Joining the observation in the previous paragraph and \Cref{patterns-with-verlappings}, we can conclude the following. If a pattern $p$ with support $B(6r,37r)$ is $g$-overlapping, then the values of $p$ on $hB(r)$ are determined by the values of $p$ outside $hB(r)$, that is, on $B(6r,37r)\smallsetminus hB(r)$. That is, every pattern in $M_{r}(g)$ is determined by its values on $B(6r,37r)\smallsetminus hB(r)$. We have proved the bound  
\[
|M_{r}(g)|\leq|L_{B(6r,37r)\smallsetminus hB(r)}(X)|.
\]
Now we use the fact that $K$ is a strong irreducibility constant for $X$, and $K\subset B(k)$. Since $hB(r-k)K\cap (B(6r,37r)\smallsetminus hB(r))=\varnothing$, we have 
\[
 |L_{B(6r,37r)\smallsetminus hB(r)}(X)|\cdot|L_{hB(r-k)}(X)| \leq |L_{B(6r,37r)}(X)|.
\]
Combining the last two inequalities and using that $|L_{hB(r-k)}(X)|=|L_{B(r-k)}(X)|$ we obtain
\begin{equation}\label{eq:upper-bound-for-_r(g)}
|M_{r}(g)|\leq\frac{|L_{B(6r,37r)}(X)|}{|L_{hB(r-k)}(X)|} = \frac{|L_{B(6r,37r)}(X)|}{|L_{B(r-k)}(X)|}.
\end{equation}
We now define $M_r$ as the set of all patterns in $L_{B(6r,37r)}(X)$ that are $g$-overlapping for some $g\in B(2r,38r)$. That is, $M_r$ equals the union of $M_r(g)$, with $g$ ranging over $B(2r,38r)$. \Cref{eq:upper-bound-for-_r(g)} shows that 
\begin{align*}
\frac{|M_{r}|}{|L_{B(6r,37r)}(X)|}  \leq & \sum_{g\in B(2r,38r)}\frac{|M_{r}(g)|}{|L_{B(6r,37r)}(X)|}\\ \leq & \sum_{g\in B(2r,38r)}\frac{1}{|L_{B(6r,37r)}(X)|} \frac{|L_{B(6r,37r)}(X)|}{|L_{B(r-k)}(X)|} \\
  \leq &\frac{|B(38r)|}{|L_{B(r-k)}(X)|}.
\end{align*}
But one hypothesis in \Cref{prop:non-overlapping-multiuso} is that $|B(38r)|<|L_{B(r-k)}(X)|$. Thus we have proved 
\[
\frac{|M_{r}|}{|L_{B(6r,37r)}(X)|} <1.
\]

It follows that the proportion of elements in $L_{B(6r,37r)}(X)$ that do not satisfy the desired condition (being non-$g$-overlapping for all $g\in B(2r,38r)$) is strictly smaller than 1. This proves our claim.
\end{proof}

\begin{lemma}[Mid range]\label{lem:mid-range-markers}
Let $G$, $X$, $K$, $k$ and $r$ be as in \Cref{prop:non-overlapping-multiuso}. Then there is a pattern $p\in L_{B(r,5r)}(X)$ which is non-$g$-overlapping for all $g\in B(2r)\smallsetminus K^3$.\end{lemma}

\begin{proof}
Note that non-triviality of $X$ and the inequality $r>16k|K^3|+2k$ implies that $K\subset K^3 \subset B(r)$. We start by constructing two patterns $p_{1}$ and $p_{2}$ in $L(X)$ with
support $B(2r)$, and with the property that for every $g\in B(2r)\smallsetminus K^3$ at least one of $p_1$, $p_2$ is non-$g$-overlapping.

In order to construct $p_1$ and $p_2$, we fix a pattern $q\in L(X)$ with support $K^2$, and let $\mathcal{H}=\{h_{1},\dots,h_{n}\}$ be a subset of $B(2r)\smallsetminus K^3$ with the property that $\{h_{i}K:i=1,\dots,n\}$ is a collection of disjoint
subsets of $G$, and $\mathcal{H}$ is maximal for inclusion. Since $K^3\subset B(r)$
and $\diam(G)>2r$, we have that $B(2r) \smallsetminus K^3\ne\varnothing$, and thus $\mathcal H$ is nonempty. Since $X$
is strongly irreducible and $|X|>1$, there are at least two different
symbols $a_{1},a_{2}$ in $L_{\{\Id_G\}}(X)$. The fact that $K$ is a strong irreducibility constant for $X$ shows that for each $i\in \{1,2\}$ there is a pattern $p_{i}\in L(X)$ with support $B(2r)$, with $p_i|_{K^2}=q$,
and with $p_i(h)=a_{i}$ for all $h\in\mathcal{H}$. 

Let us observe that for every $g\in B(2r)\smallsetminus K^3$, there is $g'\in K^2$ such
that $gg'\in\mathcal{H}$. Indeed, since $\mathcal{H}$ was chosen maximal
for inclusion, for every $g\in B(2r)\smallsetminus K^3$ there must exist $h\in\mathcal{H}$
with $gK\cap hK\ne\varnothing$. Since $K=K^{-1}$ by hypothesis, this implies that $h=gg'$ for some $g'\in K^2$.  

We now verify that the patterns $p_{1}$ and $p_{2}$ satisfy
the desired condition. Let $g\in B(2r)\smallsetminus K^3$ and let $g'$ as in the previous paragraph. Since $a_1\ne a_2$, at least one of them must be different from $q(g')$. Let $i \in \{1,2\}$ such that $a_i\ne q(g')$. Then $p_{i}(gg')=a_{i}$, while $gp_{i}(gg')=p_{i}(g')=q(g')\ne a_{i}$.
It follows that $p_{i}$ and $gp_{i}$ have different values at $gg'$ and consequently $p_i$ is non-$g$-overlapping. That is, $p_1$ and $p_2$ have the desired property. 

We will now combine $p_1$ and $p_2$ into a single pattern. Observe that, since $\diam(G)>6r$, there are two group elements $g_1$ and $g_2$ such that $|g_1|=|g_2|=3r$ and $d(g_1,g_2)=6r$. 


The hypotheses in \Cref{prop:non-overlapping-multiuso} ensure that $r\geq k+1$ and $K\subset B(k)$. It follows that $g_1B(2r)K\cap g_2B(2r)=\varnothing$. Thus by strong irreducibility there is a pattern $p\in L_{B(r,5r)}(X)$
with $g_{1}p_{1}\sqsubset p$ and $g_{2}p_{2}\sqsubset p$. \Cref{monotony-of-overlappings} implies that $p$ is non-$g$-overlapping for every $g\in (2r)\smallsetminus K^3$.
\end{proof}

\begin{lemma}[Short range]
\label{lem:short-range-markers}
Let $G$, $X$, $K$, $k$ and $r$ be as in \Cref{prop:non-overlapping-multiuso}. Then there is a pattern $p\in L(X)$ with support
$B(8k|K^3|)$ and which is non-$g$-overlapping for all $g\in K^3\smallsetminus\fix(X)$. 
\end{lemma}

\begin{proof}
We start with a general observation. Let $F\Subset G$ with $\Id_G\in F$. Observe that for every $g\in F\smallsetminus\fix(X)$, there is a pattern $p\in A^{F}$ which is non-$g$-overlapping. Indeed, if $g\not\in\fix(X)$ then there is $x\in X$ with $x\ne gx$. By
shifting $x$ we can assume $x(g)\ne x(\Id_G)$. We claim that $p=x|_{F}$ is non-$g$-overlapping. It suffices to note that $p$ and $gp$ have different values at $g\in F\cap gF$, since $p(g)=x(g)$ and $gp(g)=x(\Id_G)$. 

We will apply the observation in the previous paragraph to the set $K^3$, which contains $\Id_G$ as $X$ is non-trivial.  Let $\{g_{0},\dots,g_{n-1}\}=K^3\smallsetminus\fix(X)$. By the previous observation, for each $i\in \{0,\dots,n-1\}$ there is a
pattern $p_{i}\in L(X)$ with support $K^3$ and such that $p_i$ is $g_i$-non-overlapping. We can also choose a group element $h_{i}$ with $|h_{i}|=8ki$. Observe that in this manner $h_iK^3 \subset B(8k|K|^3)$ for each $i\in\{0,\dots,n-1\}$. 

Finally, as $k\geq 1$ and $K^7\subset B(7k)$, it follows that for every pair of distinct $i,j$ in $\{0,\dots,n-1\}$ we have $h_iK^3K\cap h_jK^3=\varnothing$ and thus the family $(h_iK^3)_{i=1}^{n-1}$ is $K$-disjoint. Since $K$ is a strong irreducibility constant for $X$, it follows that there is $p\in L(X)$ with support $B(8k|K^3|)$ such that $p|_{h_iK^3}=h_{i}p_{i}$ for $i\in \{0,\dots,n-1\}$. \Cref{monotony-of-overlappings} shows that $p$ is non-$g$-overlapping for all $g\in K^3\smallsetminus\fix(X)$.
\end{proof}

We are now ready to prove \Cref{prop:non-overlapping-multiuso}
\begin{proof}[Proof of \Cref{prop:non-overlapping-multiuso}]
Let $G$, $X$, $K$, $k$ and $r$ be as in \Cref{prop:non-overlapping-multiuso}.
\begin{enumerate}
    \item By \Cref{lem:long-range-markers} there is a pattern $p_l\in L(X)$ with support $B(6r,37r)$ that is non-$g$-overlapping for all $g\in B(2r,38r)$.
    \item By \Cref{lem:mid-range-markers}  there is a pattern $p_m\in L(X)$ with support $B(r,5r)$ that is non-$g$-overlapping for $g\in B(2r)\smallsetminus K^3$.
    \item By \Cref{lem:short-range-markers} there is a pattern $p_s\in L(X)$ with support $B(8k|K^3|)$ that is non-$g$-overlapping for all $g\in K^3\smallsetminus\fix(X)$.
\end{enumerate}

The last hypothesis on $r$ from \Cref{prop:non-overlapping-multiuso} states that $r>16k|K^3|+2k$. In particular, as $\diam(G)>36r$, there exists $h\in G$ with $|h|=5r+k+8k|K^3| < 6r$. Furthermore, it follows from $r-2k>2\cdot 8k|K^3|$ that the support of $hp_{s}$ is contained in $B(5r+k, 6r-k)$. Since $K$ is a strong irreducibility constant for $X$ and $K\subset B(k)$, it follows that there is a pattern in $L(X)$ with support
$B(r,37r)$ and having $p_{m}$, $hp_{s}$, and $p_{l}$ as sub-patterns. Since $\fix(X)\subset B(k)$ (\Cref{lem:if-X-SI-then-Fix-contained-in-K}), the pattern $p$ is non-$g$-overlapping for every $g\in B(38r)\smallsetminus \fix(X)$. That is, $p$ is a $(B(r),B(37r))$-marker. 
\end{proof}
\subsection{Proof of \Cref{thm:markers}}\label{subsec:proof-of-markers}

Having proved \Cref{prop:non-overlapping-multiuso} we are now ready to prove \Cref{thm:markers}. We start by proving the claim regarding finitely generated infinite groups.

\begin{proposition}
\label{lem:most-of-patterns-are-pseudo-markers}\label{thm:markers-for-G-fg}Let $G$ be a finitely
generated infinite group endowed with a word metric, and let $X$ 
be a nonempty strongly irreducible subshift. Then $X$ admits $(B(r),B(37r))$-markers for all large enough $r$. 
\end{proposition}
\begin{proof}
If $X$ is trivial then $\fix(X)=G$ and the statement is trivially satisfied. Suppose $X$ is non-trivial. Let $K$ and $k$ be as in \Cref{prop:non-overlapping-multiuso}. It suffices to prove that every large enough $r$ satisfies conditions (1), (2) and (3) in \Cref{prop:non-overlapping-multiuso}. As $G$ is infinite, conditions (2) and (3) are obviously satisfied. In order to show that the first condition is satisfied for large enough $r$, we prove that:
\begin{equation}\label{eq:cota-para-bolas-y-lenguajes-G-infinito-fg}
\lim_{r\to\infty}\frac{|B(38r)|}{|L_{B(r-k)}(X)|}=0.
\end{equation}

On the one hand, we have that for every $r\geq k$, \[|B(38r)|\leq  |B(r)|^{38} \leq |B(k)|^{38}\cdot |B(r-k)|^{38}.\]

On the other hand, non-triviality of $X$ and~\Cref{prop:lenguaje-de-un-subshit-SI-tiene-muchos-elementos} imply that
\begin{equation*}\label{eq:SI-implica-mucho-lenguaje}
{|L_{B(r-k)}(X)|}\geq 2^{\frac{|B(r-k)|}{2|K|}}
\end{equation*}
Thus we have
\begin{align*}
    \frac{|B(38r)|}{|L_{B(r-k)}(X)|}  \leq  |B(k)|^{38}\cdot |B(r-k)|^{38} \cdot 2^{-\frac{|B(r-k)|}{2|K|}}
\end{align*}
Since $G$ is infinite, $|B(r-k)|$ is not bounded and so the right hand side goes to $0$. We conclude that~\Cref{eq:cota-para-bolas-y-lenguajes-G-infinito-fg} holds for large enough $r$. \end{proof}

Next we prove~\Cref{thm:markers} for locally finite groups. We will need the following basic observation.

\begin{observation}\label{markers-can-be-recovered-from-restricted-subshift}
    Let $G$ be a group, let $X$ be a strongly irreducible $G$-subshift, and let $H\leqslant G$ which contains a strong irreducibility constant for $X$. Let $Y\subset W\Subset H$. Then a pattern is a $(Y,W)$-marker for $X$ if and only if it is a $(Y,W)$-marker for $X|_H = \{x|_H : x \in X\}$.
\end{observation}

    \Cref{markers-can-be-recovered-from-restricted-subshift} follows from the definition of marker and the fact $\fix(X)=\fix(X|_H)$ (see \Cref{prop:basic-things-SI}).

\begin{proposition}\label{thm:markers-but-for-locally-finite-group}
Let $G$ be an infinite and locally finite group, and let $X\subset A^{G}$
be a nonempty strongly irreducible subshift. Then $X$ is markable.
\end{proposition}

\begin{proof}
If $X$ is trivial, so is the conclusion. Let $X$ be a nontrivial strongly irreducible subshift, and let $K$ be a strong irreducibility constant for $X$ such that $K$ is symmetric (and $\Id_G\in K$). \Cref{markers-can-be-recovered-from-restricted-subshift} shows that in order to prove that $X$ is markable, it is sufficient to prove that for every $Y\Subset G$ the following holds. There is a finite group $H\leqslant G$ which contains both $Y$ and $K$, there is a set $W\subset H$ such that $Y\subset W$, and there is a $(Y,W)$-marker for the subshift $X|_H$. We will prove that this is the case by defining a sequence $(H_n)_{n\in\N}$ of finite subgroups of $G$, and applying \Cref{prop:non-overlapping-multiuso} to $H_n$ for some $n$ large enough.

Before starting with the construction let us make a notational remark. Each time that we pick a finite group $H\leqslant G$, we also endow $H$ with a finite and symmetric generating set, which then determines a metric. We always denote by $B^H(r)$ the ball in $H$ of radius $r$ in the corresponding metric. 

We now proceed with the argument. Fix a nonempty set $Y\Subset G$. We define a sequence $(S_n)_{n\geq 0}$ of finite subsets of $G$ recursively. We set $S_0=Y\cup K$. Assuming that $S_n$ has been defined, we define $S_{n+1}=S_n\cup \{g_{n+1}\}$ where $g_{n+1}$ is any element of $G$ that is not in the subgroup of $G$ generated by $S_n$. For each $n$ we let $G_n$ be the subgroup of $G$ generated by $S_n$, and we endow $G_n$ with the metric determined by the symmetric generating set $S_n\cup S_n^{-1}$. 

Let us note that every $G_n$ is finite because $G$ is locally finite. In order to apply \Cref{prop:non-overlapping-multiuso} we will need some computations. The first ingredient for the argument is that the generating set chosen grows linearly with $n$, while the cardinality of $G_n$ grows exponentially with $n$. More precisely

\begin{claim}
\label{claim:balls-in-G_n-grow-exponentially}Let $n\geq 1$ and $r\leq n$.
Then we have $|B^{G_{n}}(r)|\geq2^{r}|Y\cup K|$.
\end{claim}


Indeed, note that for every $n\geq0$ and $r\geq 0$, since $g_{n+1}\notin G_n$, we have that $B^{G_{n}}(r)\cap g_{n+1}B^{G_{n}} (r)=\varnothing$. Next, observe that $|B^{G_{n+1}}(r+1)|\geq2|B^{G_{n}}(r)|$. This follows from the fact that $B^{G_{n+1}}(r+1)$ contains both
$B^{G_{n}}(r)$ and $g_{n+1}B^{G_{n}}(r)$, and these sets are disjoint
by the previous observation. Now let $n \geq 1$ and $r\leq n$. Iterating the previous relation $r-1$ times we get
\[ |B^{G_{n}}(r)| \geq 2^{r-1}|B^{G_{n-r+1}}(1)| \geq 2^{r-1}|B^{G_{1}}(1)|   \geq 2^r|Y \cup K|.   \]
Which proves the claim.

Observe that choosing $r=n$ in~\Cref{claim:balls-in-G_n-grow-exponentially} yields that $|G_{n}|\geq 2^{n}|Y\cup K|$ for every $n\geq 1$. The last computation that we need concerns the diameters of the groups $G_n$ with respect to their respective metrics. 

\begin{claim}
\label{claim:group-n^2-has-ball-of-radiu-n} There exists $n_{0}\geq 0$ such
that for all $n\geq n_{0}$ we have $\diam(G_{n^{2}}) > n$.
\end{claim}
Indeed, since $S_{n^2}\cup S_{n^2}^{-1}$ has at most $2|Y\cup K|+2n^2$ elements, we have $|B^{G_{n^{2}}}(n)|\leq(2|Y\cup K|+2n^{2})^{n}$. On the other hand \Cref{claim:balls-in-G_n-grow-exponentially},  implies 
that $|G_{n^{2}}|\geq2^{(n^{2})}|Y\cup K|$. 
Thus we have:

\begin{align*}
\frac{|B^{G_{n^{2}}}(n)|}{|G_{n^{2}}|} & \leq\frac{(2|Y\cup K|+2n^{2})^{n}}{2^{(n^{2})}|Y\cup K|}\\
 & =\frac{1}{|Y\cup K|}\left(\frac{|Y\cup K|+n^{2}}{2^{n-1}}\right)^{n}
\end{align*}

Since $\lim_{n\to\infty}\frac{|Y\cup K|+n^{2}}{2^{n-1}}=0$, it follows that $\lim_{n\to\infty}\frac{|B^{G_{n^{2}}}(n)|}{|G_{n^{2}}|}=0$. Thus 
 \Cref{claim:group-n^2-has-ball-of-radiu-n} holds.

Now for every $n\geq n_0$ we define $H_{n}=G_{(38n)^{2}}$, and we consider the subshift $X|_{H_n}$ on $H_n$. Our goal is applying \Cref{prop:non-overlapping-multiuso} to $X|_{H_n}$ for some $n$ large enough. Note that $H_{n}$ comes endowed with a word metric, which corresponds to a symmetric generating set of cardinality at most $2|Y\cup K|+2(38n)^2$ and which contains $Y\cup K$.

It will be important in the next part that some elements do not vary with $n$. By the observation at the beginning of this proof, $K$ is a strong irreducibility constant for each $X|_{H_n}$. Furthermore, the value $k$ from \Cref{prop:non-overlapping-multiuso} can be chosen uniformly as $k=1$ because each generating set for $H_n$ contains $K$. We claim that with these choices and with $r=n$, all conditions from~\Cref{prop:non-overlapping-multiuso} are fulfilled for some $n$ large enough. That is:
\begin{enumerate}[(1)]
    \item $|B^{H_n}(38n)|<|L_{B^{H_n}(r_n-1)}(X|_{H_n})|$.
    \item $\diam(H_n) > 38n$.
    \item $n >16|K|^3+2$.
\end{enumerate}
Condition (2) follows directly from the choice $H_n=G_{38n^2}$ and \Cref{claim:group-n^2-has-ball-of-radiu-n}. Condition (3) holds for all $n$ large enough simply because $K$ and $k=1$ do not depend on $n$. Since $L_{B^{H_{n}}(n-1)}(X|_{H_n})=L_{B^{H_{n}}(n-1)}(X)$, it suffices to prove
\begin{equation}
\lim_{n\to\infty}\frac{|B^{H_{n}}(38n)|}{|L_{B^{H_{n}}(n-1)}(X)|}=0.\label{eq:lim-goal}
\end{equation}
For this we will need the previous computations about the groups $(G_n)_{n\in\N}$. Nontriviality of $X$ and~\Cref{prop:lenguaje-de-un-subshit-SI-tiene-muchos-elementos} shows that 
\begin{align}
\frac{|B^{H_{n}}(38n)|}{|L_{B^{H_{n}}(n-1)}(X)|} & \leq|B^{H_{n}}(38n)|\cdot2^{-\frac{|B^{H_{n}}(n-1)|}{2|K|}}.\label{eq:cota}
\end{align}


Since the metric for $H_{n}$ is determined by the generating set $S_{(38n)^{2}}\cup (S_{(38n)^{2}})^{-1}$, and this set has cardinality at most $2|Y\cup K|+2(38n)^{2}$, we have $|B^{H_{n}}(38n)|\leq(2|Y\cup K|+2(38n)^{2})^{38n}$. Furthermore, \Cref{claim:balls-in-G_n-grow-exponentially} shows that $|B^{H_{n}}(n-1)|\geq2^{n-1}|Y\cup K|$. It follows that  \Cref{eq:cota} can be rewritten as
\[\frac{|B^{H_{n}}(38n)|}{|L_{B^{H_{n}}(n-1)}(X)|} \leq 
(2|Y\cup K|+2(38n)^{2})^{38n}\cdot2^{-2^{n-2}\frac{|Y\cup K|}{|K|}}.
\]
Since \[
\lim_{n\to\infty}38n\log_{2}(2|Y\cup K|+(38n)^{2})-2^{n-2}\frac{|Y\cup K|}{|K|} =-\infty,
\]
We obtain that~\Cref{eq:lim-goal} holds for large enough $n$.




Take $n$ large enough such that all hypotheses of~\Cref{prop:non-overlapping-multiuso} are verified for the $H_n$-subshift $X|_{H_n}$ (and the parameters $K$, $k$, and $r=n$ defined before). \Cref{prop:non-overlapping-multiuso}  shows that $X|_{H_n}$ admits a $(B^{H_n}(n),B^{H_n}(37n))$-marker. Since $Y\subset B^{H_n}(n)$, it follows by~\Cref{obs:markers-can-be-completed} that $X|_{H_n}$ also admits a $(Y,B^{H_n}(37n))$-marker. As explained at the beginning of the proof, this pattern is also a $(Y,B^{H_n}(37n))$-marker for the subshift $X$. Finally, since $Y$ is arbitrary, this proves our claim that $X$ is markable.
\end{proof}

Finally, we show~\Cref{thm:markers} in full generality.
\begin{proof}[Proof of \Cref{thm:markers}]
Let $G$ be an infinite group, and let $X$ be a strongly irreducible subshift. If $|X|=1$ then the result is true for the trivial reason that $\fix(X)=G$. We suppose from now on that $|X|>1$. It is clear that exactly one of the following three cases occur:
\begin{enumerate}
    \item $G$ is finitely generated.
    \item $G$ is not finitely generated, but some finitely generated subgroup of $G$ is infinite. 
    \item All finitely generated subgroups of $G$ are finite, so $G$ is infinite and locally finite. 
\end{enumerate}
If $G$ is finitely generated then \Cref{thm:markers-for-G-fg} shows that for all $r$ big enough $X$ admits $(B(r),B(37r))$-markers, and this implies that $X$ is markable by~\Cref{prop:alternative-notions-for-markers}. If $G$ is locally finite then the claim follows directly from \Cref{thm:markers-but-for-locally-finite-group}.

Thus it only remains to consider the case where $G$ is not finitely generated, but some finitely generated subgroup $H\leqslant G$ is infinite. Let $K\Subset G$ be a strong irreducibility constant for $X$, and let $Y\Subset G$ be an arbitrary finite set. Consider the subgroup $H'=\langle Y,K,H\rangle \leqslant G$ generated by $H$, $Y$ and $K$. Thus $H'$ is finitely generated and infinite. By~\Cref{prop:basic-things-SI}, the subshift $X|_{H'}=\{x|_{H'}:x\in X\}$ is strongly irreducible with constant $K$. Since we already proved that the result holds for finitely generated infinite groups, it follows that there is $W\Subset H'$ and a $(Y,W)$-marker $p$ for $X'
$. \Cref{markers-can-be-recovered-from-restricted-subshift} shows that $p$ is also a $(Y,W)$-marker for $X$. Since the set $Y$ is arbitrary, we have proved that $X$ is markable.   
\end{proof}

%
%

\section{Egg markers}\label{sec:buebito}

In this section we construct finite collections of patterns which we call ``egg markers''. The main motivation is that these collections can be used to construct automorphisms of subshifts in a very explicit manner.

\begin{definition}\label{def:egg_markers_for_SI_no_metric}
    Let $G$ be a group, $Y\subset W \Subset G$ and $X\subset A^G$ be a subshift. A collection of patterns $\mathcal{E} \subset \L_W(X)$ is called a $(Y,W)$-\define{egg marker} for $X$ if it satisfies the following properties:
\begin{enumerate}
   
    \item \define{$(Y,W)$-marker:} there exists a $(Y,W)$-marker $d$ such that for every $q\in \mathcal{E}$, $q|_{W\smallsetminus Y}= d$.
    \item \define{pairwise exchangeable}: For every $x \in X$ and $q,q' \in \mathcal{E}$, we have that $x|_{G\smallsetminus W} \vee q \in X$ if and only if $x|_{G\smallsetminus W} \vee q' \in X$.
\end{enumerate}
Given a set $F\subset Y$ and $\mathcal{L}\subset A^F$, we say that $\mathcal E$ \textbf{realizes} $\mathcal{L}$ when $\mathcal{L}=\{q|_{F} : q\in\mathcal E\}$.
\end{definition}

Given a collection of egg markers $\mathcal{E}$, we want to think about the $(Y,W)$ marker $d$ as a common ``egg white'', while the restriction of the egg markers to $Y$ are different ```yolks''. The marker condition makes sure that if two of these patterns occur in some configuration, then the yolk of any of them is ``protected'' by the egg white, in the sense that it cannot overlap the other egg at all. The second condition assures that any yolk may be removed and replaced by another yolk without creating forbidden patterns. We illustrate this for $\ZZ^2$ in~\Cref{fig:buevitos}

\begin{figure}[ht!]
    \centering
    \begin{tikzpicture}[scale = 0.5]
    \foreach \i in {-10,...,10}{
        \foreach \j in {-9,...,9}{
            \draw[fill = black!10] (\i,\j) rectangle (\i+1,\j+1);
        }
    }
        \begin{scope}[shift={(0,0)}]
            \yolk
            \draw[fill = yellow!30] (2,2) rectangle (3,3);
            \draw[fill = yellow!30] (2,3) rectangle (3,4);
            \draw[fill = yellow!30] (3,2) rectangle (4,3);
            \draw[fill = yellow!30] (3,3) rectangle (4,4);  
            \node at (2.5,2.5) {$\mathtt{1}$};
            \node at (3.5,2.5) {$\mathtt{1}$};
            \node at (2.5,3.5) {$\mathtt{1}$};
            \node at (3.5,3.5) {$\mathtt{0}$};
        \end{scope}
        \begin{scope}[shift={(-5,3)}]
            \yolk
            \draw[fill = yellow!30] (2,2) rectangle (3,3);
            \draw[fill = yellow!30] (2,3) rectangle (3,4);
            \draw[fill = yellow!30] (3,2) rectangle (4,3);
            \draw[fill = yellow!30] (3,3) rectangle (4,4);
            \node at (2.5,2.5) {$\mathtt{0}$};
            \node at (3.5,2.5) {$\mathtt{1}$};
            \node at (2.5,3.5) {$\mathtt{0}$};
            \node at (3.5,3.5) {$\mathtt{1}$};
        \end{scope}

        \begin{scope}[shift={(-7,-6)}]
            \yolk
            \draw[fill = yellow!30] (2,2) rectangle (3,3);
            \draw[fill = yellow!30] (2,3) rectangle (3,4);
            \draw[fill = yellow!30] (3,2) rectangle (4,3);
            \draw[fill = yellow!30] (3,3) rectangle (4,4);
            \node at (2.5,2.5) {$\mathtt{1}$};
            \node at (3.5,2.5) {$\mathtt{0}$};
            \node at (2.5,3.5) {$\mathtt{0}$};
            \node at (3.5,3.5) {$\mathtt{1}$};
        \end{scope}

        \begin{scope}[shift={(4,-8)}]
            \yolk
            \draw[fill = yellow!30] (2,2) rectangle (3,3);
            \draw[fill = yellow!30] (2,3) rectangle (3,4);
            \draw[fill = yellow!30] (3,2) rectangle (4,3);
            \draw[fill = yellow!30] (3,3) rectangle (4,4);  
            \node at (2.5,2.5) {$\mathtt{1}$};
            \node at (3.5,2.5) {$\mathtt{1}$};
            \node at (2.5,3.5) {$\mathtt{1}$};
            \node at (3.5,3.5) {$\mathtt{1}$};
        \end{scope}
    \end{tikzpicture}
    \caption{A collection of egg markers in $\{\texttt{0},\texttt{1}\}^{\ZZ^2}$.}
    \label{fig:buevitos}
\end{figure}

\subsection{Egg models}

    Let $G$ be an infinite group, let $X\subset A^G$ be a subshift and $\mathcal{E}$ a $(Y,W)$-egg marker. We let $A_{\mathcal{E}} = \mathcal{E}\cup \{\star\}$ and consider the map $\eta_{\mathcal{E}}\colon X \to (A_{\mathcal{E}})^G$ given by
    \[  \eta_{\mathcal{E}}(x)(g) = \begin{cases}
         (g^{-1}x)|_{W} & \mbox{ if } (g^{-1}x)|_{W} \in \mathcal{E}\\
        \star & \mbox{ otherwise.}
    \end{cases}  \]

    We say that $\eta_{\mathcal{E}}(X)$ is the \define{egg-model} subshift associated to $(X,\mathcal{E})$. It is clear that $\eta_{\mathcal{E}}(X)$ is a $G$-subshift (in fact, a topological factor of $X$). The next lemma essentially says that this factor behaves like a full $G$-shift in the positions with markers.

    \begin{lemma}\label{lem:egg_exchanges_well_defined}
         Let $G$ be an infinite group, let $X\subset A^G$ be a subshift and $\mathcal{E}$ a $(Y,W)$-egg marker. For $x \in X$ let \[\Delta_x = \{ g \in G : (g^{-1}x)|_{W} \in \mathcal{E} \}.  \]
         For every map $\theta \colon \Delta_x \to \mathcal{E}$, we have that $x_{\theta}\in X$, where $x_{\theta}$ is given at $g \in G$ by \[  x_{\theta}(g) =\begin{cases}
        \bigl(\theta(h)\bigr)(w) & \mbox{ if there exists } h \in G, w \in W \mbox{ such that } g=hw \mbox{ and } (h^{-1}x)|_{W}\in \mathcal{E}. \\
        x(g) & \mbox{ otherwise.}
    \end{cases} \]
    \end{lemma}

    \begin{proof}
       First we show that $x_{\theta}$ is well defined in $A^G$. Notice that as all patterns in $\mathcal{E}$ coincide in $W\smallsetminus Y$, we have that \[  x_{\theta}(g) =\begin{cases}
        \bigl(\theta(h)\bigr)(w) & \mbox{ if there exists } h \in G, w \in Y \mbox{ such that } g=hw \mbox{ and } (h^{-1}x)|_{W}\in \mathcal{E}. \\
        x(g) & \mbox{ otherwise.}
    \end{cases} \]
    Let $g \in G$ and suppose there are $h,h'\in G$, $w,w' \in Y$ with $g = hw = h'w'$ and $q,q' \in \mathcal{E}$ such that $(h^{-1}x)|_{W}=q$ and $((h')^{-1}x)|_{W}=q'$. If we let $d$ be the common marker for $\mathcal{E}$ with support $W\smallsetminus Y$, then $(h^{-1}x)|_{W\smallsetminus Y}=d=((h')^{-1}x)|_{W\smallsetminus Y}$. It follows that $x \in h[d]\cap h'[d]\cap X.$ and in particular that \[[d]\cap (h')^{-1}h[d]\cap X  = [d]\cap w'w^{-1}[d]\cap X \neq \varnothing.\] 
    As $d$ is a $(Y,W)$-marker and $w'w^{-1}\in YY^{-1}\subset WY^{-1}$, it follows that $w'w^{-1}\in \operatorname{Fix}(X)$ and thus that $h'h^{-1} \in \operatorname{Fix}(X)$. We deduce that $q=q'$ and thus that $\bigl(\theta(h)\bigr)(w) = \bigl(\theta(h')\bigr)(w')$. Hence $x_{\theta}$ is well-defined.
    
    Next we show that $x_{\theta} \in X$. Note that for each $F\Subset G$ we have $[x_{\theta}|_F]\cap X\ne\varnothing$. This follows by applying the pairwise exchangeable condition to $x$ finitely many times (at most $|F|$ times). Thus the collection of closed sets $\{[x_{\theta}|_F]\cap X : F\Subset G\}$ has the finite intersection property, and since $X$ is compact, it has nonempty intersection. A direct verification shows that this intersection is equal to the singleton $\{x_{\theta}\}$, and thus our claim that $x_{\theta}\in X$ is proved. 
    \end{proof}

    \begin{definition}
        An automorphism $\varphi \in \Aut(\eta_{\mathcal{E}}(X))$ is called an \define{egg automorphism} if it fixes the position of stars. Namely, if \[ \varphi(\eta_{\mathcal{E}}(x))(g)=\star \mbox{ if and only if } \eta_{\mathcal{E}}(x)(g)=\star \mbox{ for every } x \in X, g \in G. \]
        We denote the space of egg automorphisms by $\Aut_{\mathcal{E}}(\eta_{\mathcal{E}}(X))$
    \end{definition}
    We note that the space $\Aut_{\mathcal{E}}(\eta_{\mathcal{E}}(X))$ of egg automorphisms is a subgroup of $\Aut(\eta_{\mathcal{E}}(X))$.

\begin{lemma}\label{lem:lema_maestro_de_los_reemplzos}
    Let $G$ be an infinite group, let $X\subset A^G$ be a subshift and $\mathcal{E}$ a $(Y,W)$-egg marker. Consider the map $\Psi_{\mathcal{E}} \colon \Aut_{\mathcal{E}}(\eta_{\mathcal{E}}(X)) \to \Aut(X)$, where for $\varphi \in \Aut_{\mathcal{E}}(\eta_{\mathcal{E}}(X))$, $x \in X$ and $g \in G$ we have 

    \[ \bigl(\Psi_{\mathcal{E}}(\varphi)\bigr)(x)(g) =  \begin{cases}
        \bigl(\varphi(\eta_{\mathcal{E}}(x))(h)\bigr)(w) & \mbox{if there exists } h \in G, w \in W \mbox{ such that } g=hw \mbox{ and } (h^{-1}x)|_{W}\in \mathcal{E}, \\
        x(g) & \mbox{otherwise.}
    \end{cases}  \]
    Then $\Psi_{\mathcal{E}}$ is a monomorphism. In particular, $\Aut_{\mathcal{E}}(\eta_{\mathcal{E}}(X)) $ embeds into $\Aut(X)$.
\end{lemma}

\begin{proof}
    Fix an egg automorphism $\varphi$ and notice by that by~\Cref{lem:egg_exchanges_well_defined}, the map $\Psi_{\mathcal{E}}(\varphi) \colon X \to X$ is well-defined. Furthermore, as $\varphi$ and $\eta_{\mathcal{E}}$ are continuous, it follows that $\Psi_{\mathcal{E}}$ is continuous. 
    
    Let $g,t \in G$. We have that 
    \begin{align*}
        \bigl(\Psi_{\mathcal{E}}(\varphi) \bigr)(tx)(g) &  = \begin{cases}
        \bigl(\varphi(\eta_{\mathcal{E}}(tx))(h)\bigr)(w) & \mbox{if } \exists h \in G, w \in W \mbox{ such that } g=hw \mbox{ and } (h^{-1}tx)|_{W}\in \mathcal{E}, \\
            (tx)(g) & \mbox{otherwise.} \end{cases} \\
            &  = \begin{cases}
            \bigl(\varphi(\eta_{\mathcal{E}}(x))(h)\bigr)(w) & \mbox{if  } \exists h \in G, w \in W \mbox{ such that } t^{-1}g=hw \mbox{ and } (h^{-1}x)|_{W}\in \mathcal{E}, \\
            x(t^{-1}g) & \mbox{otherwise.}  \end{cases} \\
            & =  \bigl(\Psi_{\mathcal{E}}(\varphi) \bigr)(x)(t^{-1}g)\\
            & =  \Bigl(t\bigl(\Psi_{\mathcal{E}}(\varphi) \bigr)(x)\Bigr)(g).
    \end{align*}
    
    It follows that $\Psi_{\mathcal{E}}(\varphi)$ is $G$-equivariant, and thus we have $\Psi_{\mathcal{E}}(\varphi)\in \operatorname{End}(X)$. Finally, a direct computation shows that for $\varphi_1,\varphi_2\in \Aut_{\mathcal{E}}(\eta_{\mathcal{E}}(X))$ then $\Psi_{\mathcal{E}}(\varphi_1 \circ \varphi_2) = \Psi_{\mathcal{E}}(\varphi_1) \circ \Psi_{\mathcal{E}}(\varphi_2)$, and thus $\Psi_{\mathcal{E}}$ is a homomorphism. In particular, $\Psi_{\mathcal{E}}(\varphi) \in \Aut(X)$ for every $\varphi \in \Aut_{\mathcal{E}}(\eta_{\mathcal{E}}(X))$.

    Finally, notice that by construction we have that for every $\varphi \in \Aut_{\mathcal{E}}(\eta_{\mathcal{E}}(X))$, then \[ \eta_{\mathcal{E}}\bigl(\Psi_{\mathcal{E}}(\varphi)(x)\bigr) = \varphi( \eta_{\mathcal{E}}(x)).   \]
    From where it follows that $\Psi_{\mathcal{E}}$ is injective.\end{proof}

    \Cref{lem:lema_maestro_de_los_reemplzos} will be the main tool to prove our embedding results. Namely, instead of embedding the automorphism group of a full shift directly on $\Aut(X)$, we do it on the space of egg automorphisms, which is much easier.

    For the proof of~\Cref{thm:Ryan} we will need a particular type of egg automorphism which arises from permutations of $\mathcal{E}$.
    
\begin{corollary}\label{cor:automorfismo_reemplazos_bien_definido}
    Let $G$ be an infinite group, let $X\subset A^G$ be a subshift and $\mathcal{E}$ a $(Y,W)$-egg marker. Let $\sigma \in \operatorname{Sym}(\mathcal{E})$ and consider the map $\phi_{\sigma}\colon X \to X$ given for every $g \in G$ by \[
        \phi_{\sigma}(x)(g) = \begin{cases}
            \Bigl(\sigma\bigl((h^{-1}x)|_{W}\bigr)\Bigr)(w) & \mbox{ if there exists } h \in G, w \in W \mbox{ such that } g=hw \mbox{ and } (h^{-1}x)|_{W}\in \mathcal{E}. \\
            x(g) & \mbox{otherwise.}
    \end{cases}\]
    The map $\phi_{\sigma} \in \Aut(X)$.
\end{corollary}

\begin{proof}
    Consider the map $\mu\colon \eta_{\mathcal{E}}(X)\to \eta_{\mathcal{E}}(X)$, where for $x\in X$ and $g \in G$, \[ \mu( \eta_{\mathcal{E}}(x))(g) = \begin{cases}
        \sigma( \eta_{\mathcal{E}}(x)(g)) & \mbox{if }  \eta_{\mathcal{E}}(x)(g) \in \mathcal{E},\\
        \star & \mbox{otherwise.}
    \end{cases}  \]
    By~\Cref{lem:egg_exchanges_well_defined} we have that $\mu$ is well-defined. Furthermore, the map $\mu$ is continuous, $G$-equivariant and fixes the position of stars, thus $\mu \in \Aut_{\mathcal{E}}(\eta_{\mathcal{E}}(X))$. By~\Cref{lem:lema_maestro_de_los_reemplzos} we have that $\Psi_{\mathcal{E}}(\mu) \in \Aut(X)$. A direct computation shows that $\phi_{\sigma} = \Psi_{\mathcal{E}}(\mu)$.\end{proof}

\subsection{Existence of egg markers}


\begin{definition}\label{def:subshift-with-many-egg-markers}
    Let $G$ be an infinite group. A $G$-subshift $X$ \define{admits complete egg markers} if for every $F\Subset G$ there exists a collection of egg markers that realizes $L_F(X)$.  
\end{definition}
        
        \begin{lemma}\label{lem:existence_egg_markers_general}
    Let $G$ be an infinite group and let $X$ be a nontrivial strongly irreducible $G$-subshift. Then $X$ admits complete egg-markers. 
    \end{lemma}
    \begin{proof}
        Fix $F\Subset G$. We will show that $X$ admits a collection of egg markers realizing $L_F(X)$. Fix a strong irreducibility constant $K$ for $X$. For $T\Subset G$ and $q \in L_{T\smallsetminus FK}(X)$ we consider the set of extensions of $q$ to $T$ given by \[ \operatorname{Ext}(q) =  \{ p \in L_{FK}(X) : p \vee q \in L_{T}(X)\}.        \]

        Observe that if $T\subset T'$, $q \in L_{T\smallsetminus FK}(X)$ and $q' \in L_{T'\smallsetminus FK}(X)$ are such that $q \sqsubset q'$, then $\operatorname{Ext}(q')\subset \operatorname{Ext}(q)$, that is, the set of extensions is monotonously decreasing under pattern extensions. Choose some $T_0\Subset G$ and $q_0 \in L_{T_0\smallsetminus FK}(X)$ such that $\operatorname{Ext}(q_0)$ is minimal for inclusion. 
        
       By~\Cref{thm:markers} we have that $X$ is markable and thus there exists $W\Subset G$ such that $T_0K \subsetneq W$ and there exists a $(T_0K,W)$-marker $d$.

        As $X$ is strong irreducible with constant $K$, there exists $x \in X$ such that $x|_{T_0 \smallsetminus FK}=q_0$ and $x|_{W \smallsetminus T_0K}=d$. Let $p_0 = x|_{W \smallsetminus FK}$. By~\Cref{obs:markers-can-be-completed} it follows that $p_0$ is an $(FK,W)$-marker. Let \[\mathcal{E} = \{q\in L_{W}(X) : q|_{W\smallsetminus FK} =p_0\}.\] 

        We take $Y=FK$ and claim that $\mathcal{E}$ is a set $(Y,W)$-egg markers which realizes $L_F(X)$. By construction, it is clear that for every $q \in \mathcal{E}$ we have $q|_{W \smallsetminus Y}=p_0$ which is a $(Y,W)$-marker. Furthermore, as $X$ is strongly irreducible and since $F$ and $W\smallsetminus FK$ are $K$-disjoint, we get that $\{q|_F : q \in \mathcal{E}\} = L_F(X)$. Thus we only need to check that $\mathcal{E}$ is pairwise exchangeable.

        Fix $q,q' \in \mathcal{E}$ and let $x \in X$ such that $x|_{W} = q$. Fix $V\Subset G$ such that $W\subset V$. As $q_0\sqsubset p_0 \sqsubset x|_{V\smallsetminus FK}$ and $\operatorname{Ext}(q_0)$ is minimal, it follows that $\operatorname{Ext}(x|_{V\smallsetminus FK}) = \operatorname{Ext}(p_0)=\operatorname{Ext}(q_0)$. Since $q'_{FK} \in \operatorname{Ext}(p_0)$, we obtain that $p_V = q'_{FK} \vee x|_{V\smallsetminus FK} \in L_V(X)$.

        Consider the collection of closed sets given by \[ \{ [p_V] \cap X : V \Subset G \mbox{ and } V \supset W\}. \]
        This collection has the finite intersection property, and thus, as $X$ is compact, it follows that it has nonempty intersection. Let $y\in X$ be an element in the intersection, by construction it follows that $y|_{G \smallsetminus FK}=x|_{G \smallsetminus FK}$ and $y|_{FK} = q'|_{FK}$, therefore $y = x_{G \smallsetminus W} \vee q'$.\end{proof}

        Next we show that under the assumption of strong TMP, we can give explicit bounds to the relative size of the support of egg markers. This will be crucial in the proof of~\Cref{thm:EmbeddingF2}.

\begin{lemma}\label{lem:existence-of-linear-egg-markers}
        Let $G$ be a finitely generated infinite group, and let $X$ be a nontrivial strongly irreducible subshift with the strong TMP. Then for all large enough integers $r$,  there is a collection of $(B(2r),B(74r))$-egg markers for $X$ that realizes $L_{B(r)}(X)$. 
\end{lemma}
\begin{proof}
Let $K$ be a strong irreducibility constant and $M$ be a memory constant associated to $X$. Take $r$ be large enough such that $K \cup M \subset B(r)$, and such that $X$ admits a $(B(2r),B(74r))$-marker $p_r$. This is possible by \Cref{thm:markers}. Then we define $\mathcal E_r$ as 
\[\mathcal{E}_r=\{q\in L_{B(74r)} : q|_{B(74r)\smallsetminus B(2r)} = p_r\}.\]
We shall prove that $\mathcal E_r$ is a collection of $(B(2r),B(74r))$-egg markers for $X$ that realizes $L_{B(r)}(X)$. The $(B(2r),B(74r))$-marker condition is obvious from the definition. Since $K\subset B(r)$ we have $B(r)K \cap (B(74r)\smallsetminus B(2r)) = \varnothing$, from where it follows that for every $p\in L_{B(r)}(X)$ there is $q\in \mathcal E_r$ with $p\sqsubset q$. Thus $\mathcal E_r$ realizes $L_{B(r)}(X)$.

Finally, we show that $\mathcal{E}_r$ is pairwise exchangeable. Let $q,q'\in \mathcal{E}_r$ and take $x \in [q]$, $y \in [q']$. By definition of $\mathcal{E}_r$, we have $x|_{B(74r)\setminus B(2r)}=y|_{B(74r)\setminus B(2r)}$. Since $X$ has the strong TMP with constant $M\subset B(r)$, it follows that if $x,y\in X$ are such that $x|_{B(3r)\setminus B(2r)}=y|_{B(3r)\setminus B(2r)}$, the element $ z= y|_{B(2r)}\lor x|_{G\smallsetminus B(2r)}$ lies in $X$. As $x = x|_{G\smallsetminus B(74r)} \vee q$ and $z = x|_{G\smallsetminus B(74r)} \vee q'$, it follows that $\mathcal{E}_r$ is pairwise exchangeable.\end{proof}
We finish the section stating a broad generalization of Ward's result  \cite{Ward1994}.

\begin{proposition}\label{prop:generalized-wards-result}
Let $G$ be an infinite group, and let $X$ be an infinite $G$-subshift that admits complete egg markers. Then every finite group embeds into $\Aut(X)$. 
\end{proposition}
\begin{proof}
    Let $n\in\N$ be arbitrary. Since $X$ is infinite we have $|L_F(X)|\geq n$ for some $F\Subset G$. Then \Cref{lem:existence_egg_markers_general} shows that there is a collection $\mathcal E$ of egg markers that realizes $L_F(X)$, and~\Cref{cor:automorfismo_reemplazos_bien_definido} shows that $\operatorname{Sym}(\mathcal{E})$ embeds into $\Aut(X)$. Since $|\mathcal E|\geq n$, we have proved that the symmetric group on $n$ elements embeds into $\Aut(X)$ for every $n\in\N$.
\end{proof}

As a consequence off~\Cref{lem:existence_egg_markers_general} and~\Cref{prop:generalized-wards-result} we get the following general result.

\begin{corollary}\label{cor:generalized_ward}
    Let $G$ be an infinite group and  $X$ be a non-trivial strongly irreducible $G$-subshift. Then every finite group embeds into $\Aut(X)$.
\end{corollary}

\section{Ryan's Theorem}\label{sec:ryan}
Recall that for $g \in G$, we denote by $\tau_g \in \Aut(A^{G})$ the right shift automorphism, that is, $\tau_g(x)(h) = x(hg)$ for every $h \in G$. 

For the next result, let us recall that we denote by $Z(G)= \{ x \in G: gx = xg \mbox{ for every } g \in G\}$ the center of $G$ and by $\operatorname{Fix}(X) = \{ g \in G : gx = x \mbox{ for every } x \in X\}$ its pointwise stabilizer. Notice that $\operatorname{Fix}(X)$ is a normal subgroup of $G$ and that $G/\operatorname{Fix}(X)$ acts faithfully on $X$. For $\gamma \in G/\operatorname{Fix}(X)$ and $x \in X$ we write $\gamma x$ to mean the action $hx$ of any $h \in G$ which belongs to the coset $\gamma=h\operatorname{Fix}(X)$.


\begin{proposition}\label{prop:egg-markers-imply-ryan}
    Let $G$ be an infinite group and let $X\subset A^G$ be a nonempty subshift which admits complete egg markers. Then $Z(\Aut(X))$ is equal to $\langle \tau_g : g\fix(X)\in Z(G/\fix(X))\rangle$ and it is isomorphic to $Z(G/\fix(X))$. 
\end{proposition}
\begin{proof}
    By \Cref{center-of-G-embeds-in-Aut(G)} it suffices to prove that every element in $Z(\Aut(X))$ is equal to $\tau_g$ for some $g \in G$ with $g\fix(X)\in Z(G/\fix(X))$. If $|X|=1$ the result holds trivially. If $|X|=2$, then $X$ consists either of an orbit of size two, in which case the conclusion holds; or $X$ consists on two fixed points, in which case $X$ does not admit complete egg markers. In what follows we will assume that $|X|\geq 3$.

    Let $\varphi \in Z(\Aut(X))$. By the Curtis-Lyndon-Hedlund theorem there exists $F\Subset G$ and a local rule $\mu \colon A^{F}\to A$ such that $\varphi(x)(g)= \mu( (g^{-1} x)|_{F})$ for every $x\in X$ and $g\in G$. We enlarge the set $F$ if necessary to ensure that $|L_F(X)|\geq 3$, and take a collection $\mathcal E$ of $(Y,W)$-egg markers which realizes $L_F(X)$. Without loss of generality we assume that $Y$ is symmetric.
    
    By~\Cref{cor:automorfismo_reemplazos_bien_definido} it follows that for every $\sigma \in \operatorname{Sym}(\mathcal{E})$ the map $\phi_{\sigma}$ which exchanges patterns from $\mathcal{E}$ according to $\sigma$ is a well-defined automorphism of $X$.

    We first claim that for every $q \in \mathcal{E}$, there exists $g \in Y$ such that $(g\varphi(x))|_{W} \in \mathcal{E}$ for every $x \in [q]\cap X$. Indeed, suppose this property does not hold for some $q \in \mathcal{E}$ and $x \in X$. It follows that for every $\sigma\in\operatorname{Sym}(\mathcal{E})$ we have
     \[ \bigl(\phi_{\sigma} \circ \varphi  (x)\bigr)(\Id_G) = \varphi(x)(\Id_G). \]
    This is simply by definition of $\phi_\sigma$. On the other hand, if we choose $\sigma$ as the transposition which exchanges $q$ with some arbitrary $q'\in \mathcal{E}$, we have  \[\bigl( \varphi \circ \phi_{\sigma}(x)\bigr)(\Id_G) = \mu(q'|_{F}). \]
    Since $\varphi \in Z(\Aut(X))$, we have $\varphi\circ \phi_\sigma (x)=\phi_\sigma\circ\varphi(x)$ and thus $\varphi(x)(\Id_G)=\mu(q'|_F)$ for every $q'\in \mathcal{E}$. This implies that $\mu|_{L_F(X)}$ is constant:  $q'\in \mathcal{E}$ is arbitrary, and, as $\mathcal{E}$ is complete, every element in $L_F(X)$ arises as the restriction of some element from $\mathcal E$ to $F$. This implies that $\varphi$ is also a constant function, and since it is an automorphism, this can only happen if $|X|=1$. This contradicts our assumption that $|X|\geq 3$.

    Now that the first claim is settled, we remark that the following holds for every  $q \in \mathcal{E}$: if there are $g,g'\in Y$ and $p,p'\in \mathcal{E}$ such that for every $x \in [q]\cap X$ we have $(g\varphi(x))\in [p]\cap X$ and $(g'\varphi(x)) \in [p']$, then necessarily $g'g^{-1}\in \operatorname{Fix}(X)$ and $p=p'$.  Indeed, $(g\varphi(x))\in [p]$ and $(g'\varphi(x)) \in [p']$ imply that $d=p|_{W\smallsetminus Y}$ is $g'g^{-1}$-overlapping. Since $d$ is a marker and $g'g^{-1}\in YY^{-1}\subset WY^{-1}$, it follows that $g'g^{-1}\in\fix(X)$. Since $g'g^{-1}$ acts trivially we have $g\varphi(x)=g'\varphi(x)$, and joining this with  our first assumption  $(g\varphi(x))\in [p]$ and $(g'\varphi(x)) \in [p']$ we obtain that $p=p'$.

    From the previous remark, it follows that there exists maps $s\colon \mathcal{E} \to G/\operatorname{Fix}(X)$ and $\pi \colon \mathcal{E} \to \mathcal{E}$ such that for every $x \in [q]$ we have $s(q)\varphi(x)\in [\pi(q)]$ and $s(q)$ is represented by some element in  $Y$. 

    We claim now that $\pi$ is the identity map. First, it is clear that $\pi$ is a permutation: the same arguments above can be applied to $\varphi^{-1}$ and the associated map must necessarily be $\pi^{-1}$. Fix $\sigma \in \operatorname{Sym}(\mathcal{E})$, $q \in \mathcal{E}$ and let $x \in [q]$. Notice that the only pattern from $\mathcal{E}$ which appears in $\phi_{\sigma}\circ \varphi(x)$ in $Y$ is $\sigma(\pi(q))$, whereas  the only pattern from $\mathcal{E}$ which appears in $\phi_{\sigma}\circ \varphi(x)$ in $Y$ is $\pi(\sigma(q))$. As $\phi_{\sigma}\circ \varphi = \varphi \circ \phi_{\sigma}$, we conclude that $\pi\circ \sigma = \sigma \circ \pi$, and thus, as $\sigma$ is arbitrary, we obtain that $\pi \in Z(\operatorname{Sym}(\mathcal{E}))$. Let us recall that the center of the symmetric group on $n$ elements contains only the trivial permutation for $n\geq 3$. Since $|\mathcal{E}|\geq 3$, it follows that $\pi = \operatorname{Id}_{\mathcal{E}}$ as claimed.

    We now argue that $s$ is a constant map. Let $p, q \in \mathcal{E}$ with $p\ne q$ and let $\sigma$ be the transposition that exchanges them. There are $g,h\in G$ such that $g,h\in Y$, $g$ represents $s(p)$ and $h$ represents $s(q)$ in the quotient $G/\fix(X)$. Now, for any $x \in [p]$ we have that $g\phi_{\sigma}(\varphi(x)) \in [q]$ and $h\varphi(\phi_{\sigma}(x)) \in [q]$. It follows that $q$ is $gh^{-1}$-overlapping, and thus $d=q|_{W\smallsetminus Y}$ is also is $gh^{-1}$-overlapping (see \Cref{monotony-of-overlappings}). Since $q$ is a $(Y,W)$-marker and $gh^{-1}\in WY^{-1}$, it follows that $gh^{-1}\in \fix(X)$. This proves that $s(q)=s(p)$.

    Let $g \in G$ be a representative of the value taken by the constant function $s$. Thus for every $q \in \mathcal{E}$ and $x \in [q]$, we have $g\varphi(x)\in [q]$. As $\varphi$ only depends locally on $F$ and $\{q|_{F} : q \in \mathcal{E}\} = L_F(X)$, it follows that $g\varphi(x) = x$ for every $x\in X$ and thus $\varphi(x) = g^{-1}x$ for every $x \in X$. This map can only be an automorphism when $g\fix(X) \in Z(G/\operatorname{Fix}(X))$ and in such case it coincides with $\tau_g$. 
    \end{proof}

    Now we are ready to prove~\Cref{thm:Ryan}.

    \begin{proof}[Proof of ~\Cref{thm:Ryan}]
    Let $G$ be an infinite group, and let $X$ be a nonempty strongly irreducible subshift. \Cref{lem:existence_egg_markers_general} implies that $X$ admits complete egg markers, thus the result follow from~\Cref{prop:egg-markers-imply-ryan}.\end{proof}

    We end this section by showing that Ryan's theorem cannot be extended to finite groups keeping the statement as it is.
    
    \begin{example}\label{ex:finite_ryan}
        Let $G$ be a finite group. Consider the subshift $X\subset \{\mathtt{0},\mathtt{1}\}^G$ given by \[ X = \{ \mathtt{0}^G,\mathtt{1}^G\} \cup \{ x \in \{\mathtt{0},\mathtt{1}\}^G : |x^{-1}(\mathtt{1})| =1 \}.  \]
        In other words, $X$ contains the two uniform configurations, and all configurations with a single $\mathtt{1}$. Note that in a finite group, every subshift is strongly irreducible (with $K=G$) and of finite type.

        It is not hard to see that $\Aut(X)$ is generated by the right shifts $\tau_g$ with $g \in G$ and the involution that exchanges $\mathtt{0}^G$ and $\mathtt{1}^G$. It follows that $\Aut(X) \cong G \times \ZZ/2\ZZ$ and thus $Z(\Aut(X)) \cong Z(G) \times \ZZ/2\ZZ$.
    \end{example}

    \subsection{Applications of Ryan's theorem}

    An immediate consequence of~\Cref{thm:Ryan} is the following

    \begin{corollary}\label{cor:hochmananswer}
        Let $G,H$ be two infinite groups and let $X\subset A^{G}$ and $Y\subset B^H$ be faithful strongly irreducible subshifts. If  $\Aut(X) \cong \Aut(Y)$ then $Z(G) \cong Z(H)$. In particular if $\min(|A|,|B|)\geq 2$ and $\Aut(A^G) \cong \Aut(B^H)$, then $Z(G)\cong Z(H)$.
    \end{corollary}

    Which answers a question of Hochman~\cite{Hochquestion}, namely, if $m,n$ are two distinct positive integers then $\Aut(\{0,1\}^{\ZZ^m})$ and $\Aut(\{0,1\}^{\ZZ^{n}})$ are non-isomorphic 

    The most famous application of Ryan's theorem is a technique to show that two full $\ZZ$-shifts whose alphabets satisfy an algebraic relation are not isomorphic. Particularly, that $\Aut(\{0,1\}^{\ZZ})$ and $\Aut(\{0,1,2,3\}^{\ZZ})$ are non-isomorphic. We generalize this result to a large class of groups.

\begin{theorem}[\Cref{thm:roots_of_shift_ingroups}]
    Let $G$ be an infinite group and suppose there exists an epimorphism $\psi\colon G \to \ZZ$ such that $\psi(Z(G))= \ZZ$. For every integer $n\geq 2$ and positive integers $k,\ell$ we have \[ \Aut( \{1,\dots,n^k\}^G) \cong \Aut( \{1,\dots,n^{\ell}\}^G) \mbox{ if and only if } k=\ell.   \]
\end{theorem}

As in the classical setting of $\ZZ$, the proof of~\Cref{thm:roots_of_shift_ingroups} relies on understanding the possible roots of elements in the center of the automorphism group. As usual, for $g\in Z(G)$ we let $\tau_g\in \aut(X)$ be the right shift by $g$. Let $G$ be a group and $X\subset A^{G}$ be a subshift.  We define \[ \mathcal{R}(X) = \{ k \geq 1 : \mbox{ for every } g \in Z(G) \mbox{ there is } \phi \in \Aut(X) \mbox{ such that } \phi^k = \tau_g\}.   \]

By~\Cref{thm:Ryan}, whenever $G$ is infinite and $X$ is a non-trivial strongly irreducible faithful subshift, then $\mathcal{R}(X)$ is the set of all possible common indices of roots of $Z(\Aut(X))$. If $X$ and $Y$ are two subshifts as above and $\Aut(X)$ and $\Aut(Y)$ are isomorphic, then by~\Cref{cor:hochmananswer} we have $\mathcal{R}(X)=\mathcal{R}(Y)$. 

We first show that if $|A|=n^k$  for some $n \geq 2$ and $k \geq 1$, then $ k \in\mathcal{R}(A^G)$. This proof is very similar to the classical argument on $G=\ZZ$.

\begin{lemma}\label{lemma:roots}
    Let $G$ be an infinite group and $A$ be an alphabet of size $|A|=n^k$, for some integers $n\geq 2$ and $k \geq 1$. Then $k \in \mathcal{R}(A^G)$.
\end{lemma}

\begin{proof}
    Without loss of generality, we can interpret the alphabet $A$ as ordered pairs of the form \[ A = \{ (a_1,\dots,a_k) : a_i \in \{1,\dots,n\} \mbox{ for every } i \in \{1,\dots,k\}\}.    \]
    Let $h \in Z(G)$. Consider the \define{slow shift} map $\phi_h \colon A^{G}\to A^G$ given by \[ \left(\phi_h(x)(g)\right)_i = \begin{cases}
        x(gh)_k & \mbox{ if } i = 1\\
        x(g)_{i-1} & \mbox{ if } 1 < i \leq k\\
    \end{cases} \mbox{ for every } g \in G. \]

    Clearly $\phi_h$ is continuous and $G$-equivariant. Furthermore, notice that for each $g\in G$ we have $ \left(\phi_h^k(x)(g)\right)_i = x(gh)_i$, thus we have $\phi_h^k = \tau_h$. From this we deduce that $\phi_h \in \Aut(A^G)$. Finally, as $h\in Z(G)$ is arbitrary, we conclude that $k \in \mathcal{R}(A^G)$.
\end{proof}

The rest of our proof is a generalization of the argument of~\cite[Theorem 8]{Lind1984} and is based on the fundamental result of Lind which characterizes the topological entropies of mixing $\ZZ$-SFTs as the class of logarithms of Perron numbers. Let us recall that a real number $\alpha>1$ is called \define{Perron}, if it is an algebraic integer which is strictly larger than the norm of its algebraic conjugates. The name comes from the fact that, as a consequence of the Perron-Frobenius theorem, these numbers are precisely the set of values that occur as the largest eigenvalue of primitive non-negative integer matrices. We refer the reader to~\cite{lind1995introduction} for further background on entropy theory and Perron numbers.

Besides the result of Lind, we will only need three basic facts about entropy theory for $\ZZ$-actions. For a compact metrizable space $X$ and an homeomorphism $T\colon X\to X$, we denote by $h_{\mbox{top}}(T)$ its topological entropy.
\begin{enumerate}
    \item Let $\sigma\colon A^{\ZZ}\to A^{\ZZ}$ be the shift map given by $\sigma(x)(n)=x(n-1)$ for every $n \in \ZZ$ and $x \in A^{\ZZ}$. The topological entropy of the shift map is given by $h_{\mbox{top}}(\sigma)=\log(|A|)$.
    \item The topological entropy is an invariant of topological conjugacy. 
    \item For every $n \geq 1$ we have $h_{\mbox{top}}(T^n)=n\cdot h_{\mbox{top}}(T)$.
\end{enumerate}

\begin{lemma}\label{lem:algebraicroots}
    Let $G$ be an infinite group and suppose there exists an epimorphism $\psi\colon G \to \ZZ$ such that $\psi(Z(G))= \ZZ$. We have that $k \in \mathcal{R}(A^G)$ if and only if $|A|$ is an $k$-th power.
\end{lemma}

\begin{proof}
    Fix a positive integer $k$. If $|A|$ is an $k$-th power, then $k \in \mathcal{R}(A^G)$ by~\Cref{lemma:roots}. Conversely, let $\psi\colon G \to \ZZ$ be an epimorphism such that $\psi(Z(G))= \ZZ$ and fix $g\in Z(G)$ such that $\psi(g)=1$. By our assumption, $k \in \mathcal{R}(A^G)$ and thus there is $\phi \in \Aut(A^G)$ such that $\phi^k = \tau_g$.

    Notice that we have an homeomorphism $\xi$ from $A^{\ZZ}$ to the set of $y \in A^G$ which are stabilized by $\operatorname{ker}(\psi)$. More precisely, this homeomorphism is given by \[ \xi(x)(h) = x(\psi(h)) \mbox{ for all } h \in G.  \]
    We consider the automorphism of $A^{\ZZ}$ given by $T = \xi^{-1}\circ \phi \circ \xi$. Notice that \[T^k = \xi^{-1}\circ \phi^k \circ \xi = \xi^{-1}\circ \tau_g 
    \circ \xi = \sigma.\]
    From where we deduce that $h_{\mbox{top}}(T) = \frac{1}{k}\log(|A|)$.

    Let $\iota_k\colon A^{\ZZ}\to A^{\ZZ}$ be given by $\iota_k(x)(n)=x(kn)$. Let $X\subset A^{\ZZ}$ be given by \[ X = \{x \in A^{\ZZ} : \iota_k(\sigma^{n+1}(x)) = T(\iota_k(\sigma^n(x))) \mbox{ for all }  n\in \{0,\dots,k-1\}\}.     \]
    
    As $T^k = \sigma$, it follows that the property which defines $X$ holds for all $n \in \ZZ$, and thus $X$ is a closed and $\ZZ$-invariant set. Furthermore, from the Curtis-Hedlund-Lyndon theorem applied to $T$, it follows that $X$ is a $\ZZ$-SFT. It is also immediate from the fact that $T^k = \sigma$ that $X$ is mixing.

    Finally, the map $\gamma \colon A^{\ZZ}\to X$ given by $\gamma(x)_{k} = T^{k \bmod{n}}(x)( \lfloor \frac{k}{n}\rfloor)$ is a topological conjugacy between $(A^{\ZZ},T)$ and $(X,\sigma)$. From where it follows that $(A^{\ZZ},T)$ is topologically conjugate to a mixing $\ZZ$-SFT. As $h_{\mbox{top}}(T) = \frac{1}{k}\log(|A|)$, it follows by Lind's result that $|A|^{\frac{1}{k}}$ must be a Perron number, and this can only occur if $|A|$ is an $k$-th power of some integer (otherwise the minimal polynomial has more than one root and they all have the same modulus). \end{proof}

    \begin{proof}[Proof of~\Cref{thm:roots_of_shift_ingroups}]
        If $k=\ell$ then obviously $\Aut( \{1,\dots,n^k\}^G) \cong \Aut( \{1,\dots,n^{\ell}\}^G)$.
        
        Conversely, if $\Aut( \{1,\dots,n^k\}^G) \cong \Aut( \{1,\dots,n^{\ell}\}^G)$ then $Z(\Aut( \{1,\dots,n^k\}^G))\cong Z(\Aut( \{1,\dots,n^{\ell}\}^G))$ and, by~\Cref{thm:Ryan}, both are isomorphic to $Z(G)$. It follows that $\mathcal{R}(\{1,\dots,n^k\}^G) = \mathcal{R}(\{1,\dots,n^\ell\}^G)$. 

        Suppose without loss of generality that $k \geq \ell$. Let $n = p_1^{\alpha_1}\cdots p_r^{\alpha_r}$ be the prime decomposition of $n$ and notice that $n$ is an $i$-th power if and only if $i$ it divides $d = \gcd(\alpha_1,\dots,\alpha_r)$. 
        
        Set $m = n^{\frac{1}{d}}$. It follows that $m \geq 2$ and that $m$ is not the $i$-th power of any positive integer for every $i \geq 2$. By~\Cref{lem:algebraicroots}, we deduce that $m^{d\ell}=n^{\ell}$ is an $dk$-th power. This implies that $m^{\ell}$ is an $k$-th power. It follows that $k$ divides $(\alpha_i /d)\ell$ for each $i=1,\dots,r $, and since  $\alpha_1 /d,\dots, \alpha_r/d$ are coprime, this implies that $k$ divides $l = \gcd((\alpha_1 /d)\ell,\dots (\alpha_r /d)\ell)$. Thus $k=l$. \end{proof}

\section{Embeddings}\label{sec:embedding}

The general objective of this section is to study when two given automorphism groups embed into each other. We make the following elementary observation that follows directly from the CHL theorem.

\begin{observation}\label{obs:trivial_embedding_when_group_and_alphabet_is_same}
     If $G,H$ are groups and $H$ embeds into $G$, then for every alphabet $A$ we have that $\Aut(A^H)$ embeds into $\Aut(A^G)$.
\end{observation}

Thus our focus will be on the case where either $H$ does not embed into $G$, or $A^G$ is replaced by a more general subshift. We will begin by proving a ``toy case''. The purpose of this is to familiarize the reader with conveyor belts and to provide intuition for the results later on in this section. The next result will not be used afterwards and is implied by~\Cref{thm:EmbeddingF2}, thus experts can safely skip to the next subsection.

\begin{proposition}\label{prop:embed_Z_easy}
Let $G$ be a non locally finite group and let $S\Subset G$ be a symmetric set such that $\langle S\rangle$ is infinite. For every pair of alphabets $A,B$ with $|B|\geq |A|^2|S|^2$ we have that $\Aut(A^{\ZZ})$ embeds into $\Aut(B^G)$.
\end{proposition}

\begin{proof}
    Choose $B_0\subset B$ with $|B_0| = |A|^2|S|^2$ and fix a bijection $\Phi\colon B_0 \to S^2 \times A^2$. For $c \in B_0$ with $\Phi(c) = (s,t,a,b)$, we write $\mathfrak{b}(c)=s$, $\mathfrak{f}(c)=t$, $\top(c)=a$, $\bot(c)=b$. The first two coordinates represent pointers backward and forward respectively, while the last two coordinates represent symbols at a top and bottom ``track'' respectively.
    
    Given $x\in X$ and $g \in G$ we say that $x$ at $g$ is
    \begin{enumerate}
        \item \define{forward consistent}, if $x(g) \in B_0$, $x(g\mathfrak{f}(x(g)) \in B_0$ and furthermore,
        \[\mathfrak{f}(x(g))^{-1} = \mathfrak{b}( x( g\mathfrak{f}(x(g)) ) .\]
        \item \define{backward consistent}, if $x(g) \in B_0$, $x(g\mathfrak{b}(x(g)) \in B_0$ and furthermore,
        \[\mathfrak{b}(x(g))^{-1} = \mathfrak{f}( x( g\mathfrak{b}(x(g)) ) .\]
    \end{enumerate}

    If the consistency conditions are satisfied, then following a pointer forward, and then backward in the new element (or vice-versa), amounts to no movement at all. We illustrate this in~\Cref{fig:well-formed}.

\begin{figure}[ht!]
\centering
\begin{tikzpicture}
    \begin{scope}[shift={(0,0)}] 
    \node at (0.5, 1) {$(s,t,\cdot,\cdot)$};
    \draw[fill=white] (0.5,-1) circle (0.4);
    \draw[color=gray, thick, dashed, ->] (0.5,0.5) to (0.5,-0.5);
    \node at (0.5,-1) {$g$};
    \draw[->, thick, bend right, shorten >=0.5cm,shorten <=0.5cm] (0.5,-1) to (2.5,-1);
    \draw[->, thick, bend right, shorten >=0.5cm,shorten <=0.5cm] (2.5,-1) to (0.5,-1);
    \node at (1.5,-1.5) {$t$};
    \node at (1.5,-0.5) {$t^{\text{-}1}$};
    \end{scope}
\begin{scope}[shift={(2,0)}] 
    \node at (0.5, 1) {$ (t^{\text{-}1},\cdot,\cdot,\cdot)$};
    \draw[fill=white] (0.5,-1) circle (0.4);
    \draw[color=gray, thick, dashed, ->] (0.5,0.5) to (0.5,-0.5);
    \node at (0.5,-1) {$gt$};
    \end{scope}
\begin{scope}[shift={(-2,0)}] 
    \node at (0.5, 1) {$(\cdot,s^{\text{-}1},\cdot,\cdot)$};
    \draw[fill=white] (0.5,-1) circle (0.4);
    \draw[color=gray, thick, dashed, ->] (0.5,0.5) to (0.5,-0.5);
    \node at (0.5,-1) {$gs$};
    \draw[->, thick, bend right, shorten >=0.5cm,shorten <=0.5cm] (2.5,-1) to (0.5,-1);
    \draw[->, thick, bend right, shorten >=0.5cm,shorten <=0.5cm] (0.5,-1) to (2.5,-1);
        \node at (1.5,-1.5) {$s^{\text{-}1}$};
        \node at (1.5,-0.5) {$s$};
\end{scope}

\end{tikzpicture}

\caption{A portion of a configuration that is forward and backward consistent at $g$. The bottom row represents elements $g, gt$ and $gs$ in $G$ and the top row the symbols at those positions.}
\label{fig:well-formed}
\end{figure}
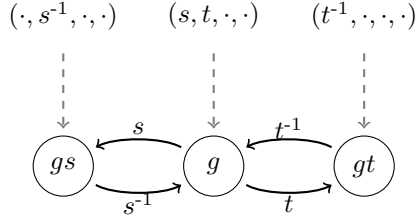

  For $x \in X$, we define the conveyor belt map $f_x \in \operatorname{Sym}(G \times \{\top,\bot\})$ as follows. Namely, for $g \in G$ we let 
    \begin{align*}
        f_x(g,\top) & = \begin{cases}
            (g\mathfrak{f}(x(g)),\top) & \mbox{if } x \mbox{ at } g \mbox{ is forward consistent}.\\
            (g,\bot) & \mbox{otherwise}.
        \end{cases}\\
        f_x(g,\bot) & = \begin{cases}
            (g\mathfrak{b}(x(g)),\bot) & \mbox{if } x \mbox{ at } g \mbox{ is backward consistent}.\\
            (g,\top) & \mbox{otherwise}.
        \end{cases}
    \end{align*}

    We remark that $f_x$ is indeed bijective, and its inverse is given by 

    \begin{align*}
        f_x^{-1}(g,\top) &= \begin{cases}
            (g\mathfrak{b}(x(g)),\top) & \mbox{if } x \mbox{ at } g \mbox{ is backward consistent}.\\
            (g,\bot) & \mbox{otherwise}.
        \end{cases}\\
         f_x^{-1}(g,\bot) &= \begin{cases}
            (g\mathfrak{f}(x(g)),\bot) & \mbox{if } x \mbox{ at } g \mbox{ is forward consistent}.\\
            (g,\top) & \mbox{otherwise}.
        \end{cases}
    \end{align*}

Intuitively, if we think of $G\times \{\top,\bot\}$ as a ``top'' and ``bottom'' track of $G$, then the $\ZZ$-action induced by $f_x$ partitions $G\times \{\top,\bot\}$ into cycles or bi-infinite paths. These possibilities are illustrated in~\Cref{fig:pathtypes}.

\begin{figure}[h!]
\centering
\begin{tikzpicture}[decoration=border]
    \begin{scope}[shift={(0,0)}] 
        \foreach \i in {-2,...,2}{
        \draw[thick, ->] (\i,0) -- (\i+1,0);
        \draw[thick, ->] (\i+1,-0.5) -- (\i,-0.5);
        }
        \draw[thick, ->] (-2,-0.5) -- (-2,0);
        \node at (3.5,0) {$\cdots$};
        \node at (3.5,-0.5) {$\cdots$};
        \node at (1,-1) {(1) backward inconsistent element};
    \end{scope}
    \begin{scope}[shift={(8,0)}] 
        \foreach \i in {-2,...,2}{
        \draw[thick, ->] (\i,0) -- (\i+1,0);
        \draw[thick, ->] (\i+1,-0.5) -- (\i,-0.5);
        }
        \draw[thick, ->] (3,0) -- (3,-0.5);
        \node at (-2.5,0) {$\cdots$};
        \node at (-2.5,-0.5) {$\cdots$};
        \node at (0.8,-1) {(2) forward inconsistent element};
    \end{scope}
    \begin{scope}[shift={(0,-2)}] 
        \foreach \i in {-2,...,2}{
        \draw[thick, ->] (\i,0) -- (\i+1,0);
        \draw[thick, ->] (\i+1,-0.5) -- (\i,-0.5);
        }
        \node at (3.5,0) {$\cdots$};
        \node at (3.5,-0.5) {$\cdots$};
        \node at (-2.5,0) {$\cdots$};
        \node at (-2.5,-0.5) {$\cdots$};
        \node at (0.5,-1) {(3) no inconsistent elements};
    \end{scope}
    \begin{scope}[shift={(8,-2)}] 
        \foreach \i in {-2,...,2}{
        \draw[thick, ->] (\i,0) -- (\i+1,0);
        \draw[thick, ->] (\i+1,-0.5) -- (\i,-0.5);
        }
        \draw[thick, ->] (-2,-0.5) -- (-2,0);
        \draw[thick, ->] (3,0) -- (3,-0.5);
        \node at (0.5,-1) {(4) two inconsistent elements};
    \end{scope}
    \begin{scope}[shift={(0,0)}] 

\end{scope}
\begin{scope}[shift={(0,0)}] 

\end{scope}

\end{tikzpicture}
\caption{Path types according to $f_x$. }
\label{fig:pathtypes}
\end{figure}
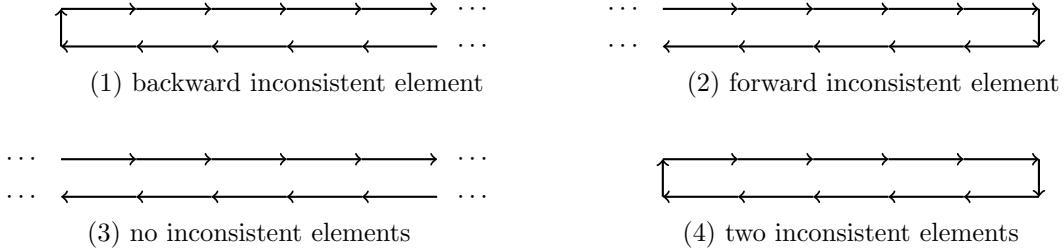

Next we use the action $\ZZ \overset{f_x}{\curvearrowright} G\times \{\top,\bot\}$ to induce configurations in $A^{\ZZ}$ by reading the top and bottom symbols respectively following $f_x$. namely, given $x \in A^G$ and $g \in G$ such that $x(g) \in B_0$, we define $c[x,g,\top],c[x,g,\bot] \in A^{\ZZ}$ which are given at $k \in \ZZ$ by \[ c[x,g,\top](k) = \begin{cases}
            \top(x(h)) & \mbox{ if } f_x^k(g,\top) = (h,\top)\\
            \bot(x(h)) & \mbox{ if } f_x^k(g,\top) = (h,\bot),\\
        \end{cases}   \]
        \[c[x,g,\bot](k) = \begin{cases}
            \top(x(h)) & \mbox{ if } f_x^k(g,\bot) = (h,\top)\\
            \bot(x(h)) & \mbox{ if } f_x^k(g,\bot) = (h,\bot).\\
        \end{cases}   \]

        We remark that for every $g,h \in G$, $x \in X$ and $t \in \{\top,\bot\}$ such that $x(h^{-1}g) \in B_0$, we have the relation \[c[hx,g,t]=c[x,h^{-1}g,t].\]

        Finally, we define $\psi\colon \Aut(A^{\ZZ}) \to \Aut(B^{G})$. For $\varphi 
        \in \Aut(A^{\ZZ})$, $x \in B^G$ and $g \in G$ we let \[\bigl(\psi(\varphi)(x)\bigr)(g) = \begin{cases}
            \Phi^{-1}\bigl(\mathfrak{b}(x(g)),\mathfrak{f}(x(g)), \varphi(c[x,g,\top])(0), \varphi(c[x,g,\bot])(0)     \bigr) & \mbox{if } x(g) \in B_0,\\
            x(g) & \mbox{otherwise.}
        \end{cases}\]

        In simpler words, $\psi(\varphi)$ is the map on $B^G$ which leaves symbols in $B\smallsetminus B_0$ unchanged, and for symbols in $B_0$ it leaves the two pointer coordinates unchanged, and updates the symbol coordinates as if it applied $\varphi$ in the copies of $\ZZ$ induced by the pointers.

        It follows from the definition that for every $\varphi\in \Aut(A^{\ZZ})$, $x \in B^G$, $g \in G$ and $t \in \{\top,\bot\}$ such that $x(g) \in B_0$, then $\varphi(c[x,g,t]) = c[\psi(\varphi)(x),g,t]$, from where it follows that for every $\varphi_1,\varphi_2 \in \Aut(A^{\ZZ})$, we have $\psi(\varphi_1 \circ \varphi_2) = \psi(\varphi_1)\circ \psi(\varphi_2)$.

        The fact that $\psi(\varphi)$ is continuous is clear as $\varphi$ is continuous. The $G$-equivariance of $\psi(\varphi)$ follows directly from the relation $c[hx,g,t]=c[x,h^{-1}g,t]$ and
        thus $\psi(\varphi)\in \operatorname{End}(B^G)$. As $\psi(\varphi^{-1}) = (\psi(\varphi))^{-1}$, we deduce that $\psi(\varphi)\in \Aut(B^G)$. Hence $\psi\colon \Aut(A^{\ZZ}) \to \Aut(B^{G})$ is a homomorphism.
        
        It only remains to show that $\psi$ is an injective map. Indeed, as the subgroup generated by $S$ is infinite, there exists a bi-infinite sequence $(g_i)_{i \in \ZZ}$ in $G$ of distinct elements, with the property that $g_i^{-1}g_{i+1}\in S$ for every $i \in \ZZ$ (this follows from K\"onig's infinity lemma applied to the Cayley graph of $\langle S\rangle $ with generating set $S$).
        
        Let $\varphi_1,\varphi_2$ be distinct automorphisms on $\Aut(A^{\ZZ})$, thus there exists $z \in A^{\ZZ}$ such that $\varphi_1(z)(0) \neq \varphi_2(z)(0)$. Fix some $a_0 \in A, b_0 \in B$ and let $\hat{x} \in B^{G}$ be the configuration given by \[ \hat{x}(g) = \begin{cases}
           \Phi^{-1}(  g_i^{-1}g_{i-1} , g_i^{-1}g_{i+1}, z(i), a_0) &\mbox{ if } g = g_i \mbox{ for some } i \in \ZZ\\
            b_0 &\mbox{ otherwise. } 
        \end{cases}  \]

        It is clear from the definition that $c[\hat{x},g_0,\top]=z$. From the definition of $\psi$ it follows that $\bigl(\psi(\varphi_1)(x)\bigr)(g_0)\neq \bigl(\psi(\varphi_2)(x)\bigr)(g_0)$ and thus $\psi(\varphi_1)\neq \psi(\varphi_2)$
        .\end{proof}

\subsection{Proof of~\Cref{thm:EmbeddingZ}}

For $k \geq 1$, we denote by $F_k$ the free group on generators $\{a_1,\dots,a_k\}$. Before proving~\Cref{thm:EmbeddingZ} we show an elementary lemma.

\begin{lemma}\label{lem:fat_free_group}
    Let $G$ be a group and suppose that $F_k$ embeds into $G$ for some $k \geq 1$. For every $T\Subset G$ there exists $\gamma_1,\dots,\gamma_k \in G$ with $\langle \gamma_1,\dots,\gamma_k\rangle \cong F_k$ and such that the collection $\{wT\}_{w \in \langle \gamma_1,\dots,\gamma_k\rangle}$ is pairwise disjoint.
\end{lemma}

\begin{proof}
    Fix $T\Subset G$ and let $\psi\colon F_k \to G$ be an injective homomorphism. Since $\psi$ is injective, there are at most $|TT^{-1}|$ elements $u$ of $F_k$ such that $T \cap \psi(u)T \neq \varnothing$. 
    
    Consider the word metric $|\cdot|$ on $F_k$ with respect to the canonical generating set $\{a_1,\dots,a_k\}$ and let $n_0 = 1+\max \{|u| : u \in F_k \mbox{ and } T \cap \psi(u)T \neq \varnothing\}$. For $i \in \{1,\dots,k\}$, take $\gamma_i = \psi(a_i^{n_0})$. Clearly $\gamma_1,\dots,\gamma_k$ generate an isomorphic copy of $F_k$ in $G$.

    Let $w,w' \in \langle \gamma_1,\dots,\gamma_k\rangle$ and suppose $wT \cap w'T \neq \varnothing$. It follows that $w^{-1}w'T \cap T \neq \varnothing$. Letting $u \in \langle a_1^{n_0},\dots, a_k^{n_0}\rangle$ such that $\psi(u) = w^{-1}w'$ we obtain that $|u| < n_0$, which can only occur if $u=\Id_{F_k}$, that is, if $w=w'$.
\end{proof}

Let $G$ be a group. Notice that $G$ admits a torsion-free element if and only if $F_1\cong \ZZ$ embeds into $G$. Also, if $F_2$ embeds into $G$ then $F_{k}$ embeds into $G$ for every $k \geq 1$. Therefore in order to prove both parts of~\Cref{thm:EmbeddingZ} it will suffice to prove that if $F_k$ embeds into $G$ for some $k \geq 1$, then $\Aut(A^{F_k})$ embeds into $\Aut(X)$.

\begin{theorem}\label{thm:embedding_Z_countable}(\Cref{thm:EmbeddingZ})
    Let $G$ be a group, $k \geq 1$ and $X$ be a non-trivial strongly irreducible $G$-subshift. If $F_k$ embeds into $G$, then $\Aut(A^{F_k})$ embeds into $\Aut(X)$ for every finite alphabet $A$.
\end{theorem}

\begin{proof}

    Fix $k \geq 1$ and consider the \define{set of tracks} $\mathfrak{T} = \{\top,\bot\}^k$. For $t = (t_i)_{1 \leq i \leq k} \in \mathfrak{T}$ and $j \in \{1,\dots,k\}$ we define the $j$-th flip $\rho_j(t)$ of $t$ by changing its $j$-th coordinate, namely \[ \bigl(\rho_j(t)\bigr)_i = \begin{cases} t_i & \mbox{ if } i \neq j\\
    \top  & \mbox{ if } i = j \mbox{ and } t_j = \bot\\ 
      \bot  & \mbox{ if } i = j \mbox{ and } t_j = \top.
    \end{cases}  \]

    Consider the alphabet $\mathcal{B} = A^{\mathfrak{T}}$ and notice that $|\mathcal{B}| = |A|^{2^k}$. Let $K\Subset G$ be a constant of strong irreducibility for $X$. As $X$ is strongly irreducible and non-trivial, there exists $F_0\Subset G$ such that $|L_{F_0}(X)|\geq |\mathcal{B}|$. By~\Cref{lem:existence_egg_markers_general}, we have that $X$ admits complete egg markers, thus there exists $Y, W\Subset G$ such that $Y$ is symmetric, contains the identity, $F_0\subset Y\subset W$ and there exists a collection $\mathcal{E}'$ of $(Y,W)$-egg markers which realizes $L_{F_0}(X)$. As $|L_{F_0}(X)|\geq |\mathcal{B}|$, it follows that $|\mathcal{E}'|\geq |\mathcal{B}|$. We fix a subset $\mathcal{E}\subset \mathcal{E}'$ with $|\mathcal{E}|=|\mathcal{B}|$. Let $A_{\mathcal{E}} = \mathcal{E}\cup \{\star\}$ and consider the map $\eta_{\mathcal{E}} \colon X \to (A_{\mathcal{E}})^G$ given by \[ \eta_{\mathcal{E}}(x)(g) = \begin{cases}
        g^{-1}x|_{W} & \mbox{if } g^{-1}x|_{W} \in \mathcal{E}\\
        \star & \mbox{otherwise.}
    \end{cases}  \]

    Thus $\eta_{\mathcal{E}}(X)\subset 
    (A_{\mathcal{E}})^G$ is the egg-model associated to $(X,\mathcal{E})$ (see~\Cref{sec:buebito}). By~\Cref{lem:lema_maestro_de_los_reemplzos}, we know that the group of egg automorphisms $\Aut_{\mathcal{E}}( \eta_{\mathcal{E}}(X))$ embeds into $\Aut(X)$, therefore to conclude it suffices to show that $\Aut(A^{F_k})$ embeds into $\Aut_{\mathcal{E}}( \eta_{\mathcal{E}}(X))$.

    Given $x \in X$ and $g \in G$, write $y_x = \eta_{\mathcal{E}}(x)$. Note that given $x \in X$, the non-$\star$ positions of $y_x$ are precisely those where a pattern from $\mathcal{E}$ occurs.

    By~\Cref{lem:fat_free_group}, there exist $\gamma_1,\dots,\gamma_k \in G$ such that $\langle \gamma_1,\dots,\gamma_k\rangle \cong F_k$ and the collection $\{uWK\}_{u \in \langle \gamma_1,\dots,\gamma_k\rangle}$ is pairwise disjoint. Given $x \in X$, $g \in G$ and $i \in \{1,\dots,k\}$, we will say that $x$ at $g$ is
    \begin{enumerate}
        \item \define{forward $i$-consistent}, if $\bigl(y_x\bigr)(g) \in \mathcal{E}$ and $\bigl(y_x\bigr)(g\gamma_i) \in \mathcal{E}$.
        \item \define{backward $i$-consistent}, if $\bigl(y_x\bigr)(g) \in \mathcal{E}$ and $\bigl(y_x\bigr)(g\gamma_i^{-1}) \in \mathcal{E}$.
    \end{enumerate}

    We remark that if $x$ is $i$-forward consistent at some $g \in G$, then $x$ is backward $i$-consistent at $g\gamma_i$. Similarly, if $x$ is $i$-backward consistent at $g$, then $x$ is forward $i$-consistent at $g\gamma_i^{-1}$. We use this property to construct a right action of $F_k$ on $G \times \mathfrak{T}$ which depends on $x$. More precisely, for $x \in X$ and $i \in \{1,\dots,k\}$ we let the $i$-th conveyor belt map $f_{i,x} \in \operatorname{Sym}(G\times \mathfrak{T})$ be given by

     \[f_{i,x}(g,t) = \begin{cases}
            (g\gamma_i,t) & \mbox{ if } t_i = \top \mbox{ and } p \mbox{ at } g \mbox{ is forward $i$-consistent}.\\
            (g\gamma_i^{-1},t) & \mbox{ if } t_i = \bot \mbox{ and }  p \mbox{ at } g \mbox{ is backward $i$-consistent}.\\
            (g,\rho_i(t)) & \mbox{ otherwise}.
        \end{cases}\]

    Note that $f_{i,x}$ is indeed invertible, its inverse is given by 
        \[f_{i,x}^{-1}(g,t) = \begin{cases}
            (g\gamma_i^{-1},t) & \mbox{ if } t_i = \top \mbox{ and } p \mbox{ at } g \mbox{ is backward $i$-consistent}.\\
            (g\gamma_i,t) & \mbox{ if } t_i = \bot \mbox{ and }  p \mbox{ at } g \mbox{ is forward $i$-consistent}.\\
            (g,\rho_i(t)) & \mbox{ otherwise}.
        \end{cases}\]

    The maps ${f_{1,x},\dots,f_{k,x}}$ induce a right action $\xi_x$ of $F_k$ on $G \times \mathfrak{T}$ by letting the generator $a_i$ act by $f_{i,x}$. For $u \in F_k$ and $(g,t) \in G \times \mathfrak{T}$, we write $\xi_x(g,t,u)$ for the element of $G \times \mathfrak{T}$ obtained by this action. We observe that if $y_x(g) \in \mathcal{E}$, then for any $t \in \mathfrak{T}$ the orbit of $(g,t)$ under this action will consist uniquely of pairs $(h,t')$ such that $y_x(h) \in \mathcal{E}$. Equivalently, for every $x \in X$ the set $\bigl(y_x\bigr)^{-1}(\mathcal{E}) \times \mathfrak{T}$ is invariant under the right action of $F_k$ induced by $x$. 
    
    Next we use the observation above to induce configurations of $A^{F_k}$. Fix a bijection $\Phi \colon \mathcal{E}\to \mathcal{B}$. For $x \in X$, $g\in G$ and $t \in \mathfrak{T}$ such that $y_x(g) \in \mathcal{E}$, we define $c[x,g,t] \in A^{F_k}$ by \[    c[x,g,t](u) = \Phi\Bigl(\bigl(y_x\bigr)(h)\Bigr)(t') \mbox{ for every } u \in F_k, \mbox{ where }(h,t') = \xi_x(g,t,u).\]

    We note that for $x \in X$, $g,h \in G$ and $t \in \mathfrak{T}$ we have $c[hx,g,t]=c[x,h^{-1}g,t]$.

    Next we define a map $\psi_{\mathcal{E}}\colon \Aut(A^{F_k}) \to \Aut_{\mathcal{E}}( \eta_{\mathcal{E}}(X))$. For $\varphi \in  \Aut(A^{F_k})$, $x \in X$ and $g \in G$ we let \[ \Bigl(\bigl(\psi_{\mathcal{E}}(\varphi)\bigr)(y_x)\Bigr)(g) = \begin{cases}
        \Phi^{-1}\Bigl(\bigl(\varphi(c[x,g,t])(\Id_{F_k})\bigr)_{t \in \mathfrak{T}}\Bigr) & \mbox{ if } y_x(g) \in \mathcal{E},\\
        \star & \mbox{ if } y_x(g) = \star.
    \end{cases}  \]
    
    Intuitively, $\psi_{\mathcal{E}}(\varphi)$ leaves the star positions unchanged, and everywhere else it reads the configuration induced by $x$ on each tape, applies $\varphi$, and updates the corresponding tape symbol accordingly.

    By~\Cref{lem:egg_exchanges_well_defined} it follows that for every $\varphi \in \Aut(A^{F_k})$ and $x \in X$, then $\bigl(\psi_{\mathcal{E}}(\varphi)\bigr)(y_x) \in \eta_{\mathcal{E}}(X)$. Since $\varphi$ is continuous, it is clear that $\psi_{\mathcal{E}}(\varphi)$ is continuous. Furthermore, the relation $c[hx,g,t]=c[x,h^{-1}g,t]$ implies that $\psi_{\mathcal{E}}(\varphi)$ is a $G$-equivariant map, thus $\psi_{\mathcal{E}}(\varphi) \in \operatorname{End}(\eta_{\mathcal{E}}(X))$. 
    
    A direct computation shows that the map $\psi_{\mathcal{E}}$ induces an homomorphism from $\Aut(A^{F_k})$ to $\operatorname{End}(\eta_{\mathcal{E}}(X))$, and thus in fact $\psi_{\mathcal{E}}(\varphi) \in \Aut(\eta_{\mathcal{E}}(X))$. Finally, as the position of stars is preserved, each $\psi_{\mathcal{E}}(\varphi)$ is in fact an egg automorphism and thus we conclude that $\psi_{\mathcal{E}}\colon \Aut(A^{F_k}) \to \Aut_{\mathcal{E}}( \eta_{\mathcal{E}}(X))$ is a homomorphism.

    Let us argue that $\psi_{\mathcal{E}}$ is injective. Let $\varphi_1,\varphi_2 \in \Aut(A^{F_k})$ be distinct and denote $\bar{t}= (\top,\dots,\top)\in \mathfrak{T}$. It follows that there is $z \in A^{F_k}$ such that $\varphi_1(z)(\Id_{F_k})\neq \varphi_2(z)(\Id_{F_k})$. Fix some $a \in A$, for $h \in \langle \gamma_1,\dots,\gamma_k\rangle$ we let $b_h \in \mathcal{B}$ be given by \[ b_h(t) = \begin{cases}
        z(h) & \mbox{ if } t = \bar{t},\\
        a & \mbox{ otherwise.}
    \end{cases}  \]
    Let $q_h = \Phi^{-1}(b_h) \in \mathcal{E}$. As the collection $\{hWK\}_{h \in \langle \gamma_1,\dots,\gamma_k\rangle}$ is pairwise disjoint and $X$ is strongly irreducible with constant $K$, it follows that there exists $\hat{x} \in X$ such that for every $h \in \langle \gamma_1,\dots,\gamma_k\rangle$ then \[(h^{-1}\hat{x})|_{W}=q_h.\]
    By our definitions above, it follows that $c[\hat{x},\Id_{G},\bar{t}] = z$ and thus that 
    \begin{align*}
        \Phi\Bigg(\Bigl(\bigl(\psi_{\mathcal{E}}(\varphi_1)\bigr)(y_x)\Bigr)(\Id_{G})\Bigg)(\bar{t}) & = \varphi_1(c[\hat{x},\Id_{G},\bar{t}])(\Id_{F_k}) = \varphi_1(z)(\Id_{F_k}).\\
        \Phi\Bigg(\Bigl(\bigl(\psi_{\mathcal{E}}(\varphi_2)\bigr)(y_x)\Bigr)(\Id_{G})\Bigg)(\bar{t}) & = \varphi_2(c[\hat{x},\Id_{G},\bar{t}])(\Id_{F_k}) = \varphi_2(z)(\Id_{F_k}).
    \end{align*}
    Which implies that $\psi_{\mathcal{E}}(\varphi_1)\neq \psi_{\mathcal{E}}(\varphi_2)$, thus $\psi_{\mathcal{E}}$ is injective.\end{proof}



\subsection{Proof of~\Cref{thm:EmbeddingF2}}

The structure of this proof is very similar to the proof of~\Cref{thm:EmbeddingZ} with the caveat that in this case there might be no embedded copy of $F_k$ in $G$. This is fixed by coding instead a series of geometric moves according to some large set of generators as in the proof of~\Cref{prop:embed_Z_easy} and showing that with these moves a geometric inflated copy of $F_k$ can be realized. The main difficulty is that now we need the set of egg markers to have enough internal space in the yolk to codify this set of moves, while at the same time the set of moves grows as we take egg markers with larger support. Hence the need for strong TMP and for explicit bounds in the marker lemma.

Before going into the proof of~\Cref{thm:EmbeddingF2}, we will need to prove a technical result (\Cref{lem:fat_nonamenable}) which states that nonamenable groups can be ``inflated'' arbitrarily preserving their local structure in a coarse geometric way. To do this, we will need to recall a few classical notions.


Recall that a group $G$ is called \define{nonamenable} if there exists $T\Subset G$ and $\delta > 0$ such that for every $F\Subset G$ then \[ |FT\smallsetminus F| > \delta|F|.  \]
An elementary computation shows that if $G$ is nonamenable, then one can find a set $T_0\Subset G$ such that $|FT_0| > 2|F|$ for every $F \Subset G$ (see~\cite[Theorem 4.9.2]{ceccherini-SilbersteinC09}).

Let $G$ be a finitely generated group, $S\Subset G$ a symmetric set of generators and $m$ a positive integer. We say that $\Delta \subset G$ is
\begin{itemize}
    \item \define{$m$-separated} if for every distinct $g,h \in \Delta$ we have $d_S(g,h)\geq m$.
    \item \define{$m$-covering} if for every $h \in G$ there is $g \in \Delta$ with $d_S(g,h) \leq m$.
\end{itemize}
We remark that every maximal $m$-separated set $\Delta\subset G$ is necessarily $m$-covering. 

A \define{bipartite graph} is an undirected graph $\Gamma$ whose vertex set is the union of two disjoint sets $U$ and $V$ and all its edges are between elements of $U$ and elements of $V$. We say that $\Gamma$ is \define{locally finite} if the degree of every vertex is finite. Given a vertex $u$ in $\Gamma$, we denote by $\mathcal{N}(u)$ the set of all vertices adjacent to $u$ and for a set $A$ of vertices, we write $\mathcal{N}(A) = \bigcup_{u \in A}\mathcal{N}(u)$.

 A perfect matching is a bijection $\varphi \colon U \to V$ with the property that $(u,\varphi(u))$ is an edge for every $u \in U$. We say that $\Gamma$ satisfies \define{Hall's conditions} if for every $A\Subset U$ and $B \Subset V$ we have $|\mathcal{N}(A)|\geq |A|$ and $|\mathcal{N}(B)|\geq |B|$. Hall's matching theorem~\cite{MHall1948} states that a locally finite bipartite graph admits a perfect matching if and only if it satisfies Hall's conditions.

\begin{lemma}\label{lem:fat_nonamenable}
    Let $G$ be a finitely generated nonamenable group and let $S\Subset G$ be a symmetric generating set. There exists a constant $C\geq 1$ such that for every positive integer $m$ and $m$-covering subset $\Delta$ of $G$, there exists a bijection $\varphi \colon G \to \Delta$ such that $
    d_S(g,\varphi(g))\leq Cm$ for every $g\in G$. 
\end{lemma}

\begin{proof}
    For $n \geq 0$, denote by $B(n)$ the ball of radius $n$ with respect to the word metric generated by $S$. As $G$ is nonamenable, there exists $T_0\Subset G$  such that for every $F\Subset G$, we have $|FT_0|>2|F|$. Let $t_0$ such that $T_0 \subset B(t_0)$ and notice that $|FB(t_0)| \geq |FT_0|>2|F|$. In particular we obtain that 
    \[|FB(m t_0 \lceil \log_2(|S|)\rceil)| \geq |S|^m|F|.\]
    As $|B(m)| \leq |S|^m$, it follows that if we take $C = t_0\lceil \log_2(|S|)\rceil $ then  \[|FB(Cm)| \geq |B(m)||F|.\]
    
    Now consider the locally finite bipartite graph $\Gamma$, where the vertices are given by the disjoint union of $U=G$ and $V=\Delta$ and $\{g,h\}$ with $g \in G$ and $h \in \Delta$ is an edge if and only if $d_S(g,h)\leq Cm$. A map $\varphi$ which satisfies the requirements of the lemma is then given by a perfect matching in $\Gamma$. Thus it suffices to verify that $\Gamma$ satisfies Hall's conditions.

    For every $B\Subset\Delta$, the set $\mathcal{N}(B)\subseteq G$ is at least as large as $B$ because $(h,h)$ is an edge for every $h\in\Delta$. Now let $A \Subset G$ and remark that
    \[\mathcal{N}(A) = \Delta \cap \bigcup_{g \in A} gB(Cm) = \Delta \cap A B(Cm).\]
    Since $\Delta$ is $m$-covering, we have
    \[|\mathcal{N}(A)| = |\Delta \cap A B(Cm)| \geq \frac{|AB(Cm)|}{|B(m)|}.\]
    As $|AB(Cm)| \geq |B(m)||A|$, we obtain that $|\mathcal{N}(A)| \geq |A|$. Thus Hall's conditions are satisfied.\end{proof}

    The last ingredient is the fact that every nonamenable group contains a ``inflated geometric'' copy of the free group $F_2$. For this we use the following result of Whyte~\cite{whyte_amenability_1999} on the existence of translation-like actions of $F_2$ on nonamenable groups. We recall that a group $\Gamma$ acts translation-like on a metric space $(X,d)$ if the action is free and bounded, that is, for every $\gamma \in \Gamma$ we have that $\sup_{x \in X}d(x,\gamma\cdot x)<\infty$.

    \begin{theorem}[Theorem 6.1 of~\cite{whyte_amenability_1999}]\label{thm:Whyte1999}
        Let $G$ be a finitely generated group equipped with a word metric. Then $G$ is nonamenable if and only if $G$ admits a translation-like action by $F_2$.
    \end{theorem}

    In the next result, we use $d_{F_k}$ to denote the canonical word metric in $F_k$.

    \begin{lemma}\label{lem:fat_whyte}
        Let $G$ be a finitely generated nonamenable group,  $S\Subset G$ be a symmetric generating set, and $k\geq 1$. There exists a constant $N\geq 1$ such that for every positive integer $m$, there exists a sequence $(\gamma_u)_{u \in F_k}$ of elements of $G$ which satisfies the following properties:
        \begin{enumerate}
            \item The collection $(\gamma_u B_S(m))_{u \in F_k}$ is pairwise disjoint.
            \item $d_S(\gamma_{u},\gamma_v) \leq (4m+2+d_{F_k}(u,v))N$ for every $u,v \in F_k$
        \end{enumerate}
    \end{lemma}

    \begin{proof}
        By~\Cref{thm:Whyte1999}, $G$ admits translation-like action by $F_2$. As $F_k$ embeds into $F_2$ for every $k \geq 1$, it follows that $G$ also admits a translation-like action by $F_k$. For $u \in F_k$ denote $g_u = u \cdot \Id_{G}$ the element obtained by acting by $u$ on the identity of $G$, thus $\{g_u\}_{u \in F_k}$ denotes the orbit of the identity.

        As the above action is free, it follows that $\{g_u\}_{u \in F_k}$ does not repeat elements. Furthermore, as the action is bounded, there exists a constant $L\geq 1$ such that \[ d_S(g_u,g_v) \leq Ld_{F_k}(u,v).  \]

        Fix $m$ and let $\Delta$ be a maximal and $(2m+1)$-separated subset of $G$. An elementary argument yields that that $\Delta$ is also $(2m+1)$-covering and thus by~\Cref{lem:fat_nonamenable} there exists $C\geq 1$ and a bijection $\varphi\colon G \to \Delta$ such that $d_S(g,\varphi(g))\leq C(2m+1)$ for every $g\in G$. 

        For $u \in F_k$, let $\gamma_u = \varphi(g_u)$ and take $N = \max(C,L)$. As $\{g_u\}_{u \in F_k}$ does not repeat elements and $\varphi$ is a bijection, it follows that $\{ \gamma_u\}_{u \in F_k} $ does not repeat elements either. Furthermore, as each $\gamma_u \in \Delta$ and $\Delta$ is $(2m+1)$-separated, it follows that $\{\gamma_u B_S(m)\}_{u \in F_k}$ is pairwise disjoint. 
        
        Finally, the triangular inequality yields that for $u,v \in F_k$,
    \begin{align*}
        d_S(\gamma_u,\gamma_v) & \leq d_S(g_u,\gamma_u) + d_S(g_u,g_v) + d_S(g_v,\gamma_v)\\
        & \leq C(2m+1) + L d_{F_k}(u,v) + C(2m+1)\\
        & \leq (4m+2+d_{F_k}(u,v))N. \qedhere
    \end{align*}
    \end{proof}

    \begin{theorem}\label{thm:embeddingF2_finitely_gen}
    Let $G$ be a finitely generated group and $X$ be a non-trivial strongly irreducible $G$-subshift which satisfies the strong TMP. For every finite alphabet $A$: \begin{enumerate}[(1)]
        \item If $G$ is infinite, then $\Aut(A^{\ZZ})$ embeds into $\Aut(X)$.
        \item If $G$ is nonamenable, then $\Aut(A^{F_k})$ embeds into $\Aut(X)$ for every $k \geq 1$.
    \end{enumerate}
\end{theorem}

    \begin{proof}
        Let $K$ be a strong irreducibility constant for $X$. Fix a symmetric set of generators $S\Subset G$ with the property that $K \subset S$. For $n \geq 0$, let $B(n)$ denote the ball of radius $n$ with respect to the word metric generated by $S$ and note that $K\subset B(1)$. 
        
        In what follows, we shall briefly split the proof in two cases and then unify both branches.

        \textbf{Case 1}: if $G$ is infinite and amenable, we fix $k = 1$ and show that $\Aut(A^{F_k})$ embeds into $\Aut(X)$. Note that if the map $n \mapsto |B(n)|$ has linear growth, then $G$ is virtually $\ZZ$ and thus the conclusion follows from~\Cref{thm:EmbeddingZ}. Therefore, we may assume that the map $n \mapsto |B(n)|$ is superlinear. We choose $N=1$. 

        \textbf{Case 2}: if $G$ is nonamenable, we let $k \geq 1$ be an arbitrary positive integer and show that $\Aut(A^{F_k})$ embeds into $\Aut(X)$. In this case, the map $n \mapsto |B(n)|$ has exponential growth. We choose $N\geq 1$ as the constant from~\Cref{lem:fat_whyte} associated to $G$ and $S$.

        In both cases we can choose a positive integer $r$ such that:

        \begin{enumerate}[(a)]
        \item There is a set of $(B(2r),B(74r))$-egg markers for $X$ which realizes $L_{B(r)}(X)$.
        \item $\log(|L_{B(r)}(X)|) \geq 2^k\log\left(|A|\right) + 2k\log\left(|B((296r+7)N )|\right)$.
    \end{enumerate}

    By~\Cref{lem:existence-of-linear-egg-markers}, condition (a) holds for all $r$ large enough. For condition (b), notice that \[2^k\log\left(|A|\right) + 2k\log\left(|B((296r+7)N )|\right) \leq 2^k \bigl( \log( |A|) + (296r+7)N\log(|S|)\bigr).\] By~\Cref{prop:lenguaje-de-un-subshit-SI-tiene-muchos-elementos} and the fact that $X$ is non-trivial, we have that \[   
        \log(|L_{B(r)}(X)|) \geq \frac{|B(r)|\log(|L_{\{\Id_G\}}(X)|)}{2|K|} \geq \frac{|B(r)|\log(2)}{2|K|}.
    \]
    Thus it suffices to argue that for $r$ large enough then \[ |B(r)| \geq \frac{2^{k+1}|K| }{\log(2)} \bigl( \log( |A|) + (296r+7)N\log(|S|)\bigr). \]
    Notice that every constant is fixed except $r$. As the right hand side grows linearly with $r$, and in both cases the map $r\mapsto |B(r)|$ is superlinear, it follows that (b) holds for large enough $r$.

    Now we consider both cases at the same time. Consider the \define{set of tracks} $\mathfrak{T} = \{\top,\bot\}^k$. For $t = (t_i)_{1 \leq i \leq k} \in \mathfrak{T}$ and $j \in \{1,\dots,k\}$ we define the $j$-th flip $\rho_j(t)$ of $t$ by changing its $j$-th coordinate, namely \[ \bigl(\rho_j(t)\bigr)_i = \begin{cases} t_i & \mbox{ if } i \neq j\\
    \top  & \mbox{ if } i = j \mbox{ and } t_j = \bot\\ 
      \bot  & \mbox{ if } i = j \mbox{ and } t_j = \top.
    \end{cases}  \]

    Let $D = B((296r+7)N)$ be the set of \define{directions} and let $\mathcal{B}= D^{2k} \times A^{\mathfrak{T}}$. By condition (a) there exists a set $\mathcal{E}'$ of $(B(2r),B(74r))$-egg markers for $X$ which realizes $L_{B(r)}(X)$ (thus $|\mathcal{E}'| \geq |L_{B(r)}(X)|$), and by condition (b) we have that \[|\mathcal{E}'| \geq |L_{B(r)}(X)| \geq |A|^{2^k}|D|^{2k} = |\mathcal{B}|.\]
     Thus we may extract a subset of egg markers $\mathcal{E}\subset \mathcal{E}'$ with $|\mathcal{E}|=|\mathcal{B}|$ and fix a bijection $\Phi \colon \mathcal{E} \to \mathcal{B}$. 

     Let $A_{\mathcal{E}} = \mathcal{E}\cup \{\star\}$ and consider the map $\eta_{\mathcal{E}} \colon X \to (A_{\mathcal{E}})^G$ given by \[ \eta_{\mathcal{E}}(x)(g) = \begin{cases}
        g^{-1}x|_{W} & \mbox{if } g^{-1}x|_{W} \in \mathcal{E}\\
        \star & \mbox{otherwise.}
    \end{cases}  \]
     
    Thus $\eta_{\mathcal{E}}(X)$ is the egg-model associated to $(X,\mathcal{E})$. Since by~\Cref{lem:lema_maestro_de_los_reemplzos} the group of egg automorphisms $\Aut_{\mathcal{E}}(\eta_{\mathcal{E}}(X))$ embeds into $\Aut(X)$, it suffices to show that $\Aut(A^{F_k})$ embeds into $\Aut_{\mathcal{E}}(\eta_{\mathcal{E}}(X))$

    Given $d \in D^{2k}$, we write $d= (d_1^{-1},d_1^{+1},\dots,d_k^{-1},d_k^{+1})$. Furthermore, for $q \in \mathcal{E}$ such that $\Phi(q) = (d,w) \in D^{2k} \times A^{\mathfrak{T}}$, $i \in \{1,\dots,k\}$ and $t \in \mathcal{T}$ we write 
     \[ \delta(q) = d, \quad \mathfrak{b}_i(q) = d_i^{-1}, \quad \mathfrak{f}_i(q) = d_i^{+1}, \mbox{ and } t(q) = w(t).\]

    In order to ease the notation, given $x \in X$ and $g \in G$, we will write $y_x$ to denote $\eta_{\mathcal{E}}(x)$. We will say that $x$ at $g$ is
    
    \begin{enumerate}
        \item \define{forward $i$-consistent}, if $y_x(g) \in \mathcal{E}$, $y_x(g \mathfrak{f}_i\bigl( y_x(g)\bigr)   ) \in \mathcal{E}$ and furthermore, \[    \bigl( \mathfrak{f}_i( y_x(g)) \bigr)^{-1} = \mathfrak{b}_i(y_x(g \mathfrak{f}_i( y_x(g)))   ).  \]
        \item \define{backward $i$-consistent}, if $y_x(g) \in \mathcal{E}$, $y_x(g \mathfrak{b}_i\bigl( y_x(g)\bigr)   ) \in \mathcal{E}$ and furthermore, \[    \bigl( \mathfrak{b}_i( y_x(g)) \bigr)^{-1} = \mathfrak{f}_i(y_x(g \mathfrak{b}_i( y_x(g)))   ).  \]
    \end{enumerate}

    Now for $x \in X$ and $i \in \{1,\dots,k\}$ we let the $i$-th conveyor belt map $f_{i,x} \in \operatorname{Sym}(G\times \mathfrak{T})$ be given on $(g,t) \in G\times \mathfrak{T}$ by

     \[f_{i,x}(g,t) = \begin{cases}
            (g\mathfrak{f}_i(x(g)),t) & \mbox{ if } t_i = \top \mbox{ and } p \mbox{ at } g \mbox{ is forward $i$-consistent},\\
            (g\mathfrak{b}_i(x(g)),t) & \mbox{ if } t_i = \bot \mbox{ and }  p \mbox{ at } g \mbox{ is backward $i$-consistent},\\
            (g,\rho_i(t)) & \mbox{ otherwise}.
        \end{cases}
    \]

    Similarly as in the proof of~\Cref{thm:EmbeddingZ}, the maps $f_{i,x}$ are invertible and thus are truly elements of $\operatorname{Sym}(G\times \mathfrak{T})$, thus they induce a right $F_k$ action on $G\times \mathfrak{T}$ by letting the generator $a_i$ act by $f_{i,x}$. For $u \in F_k$ and $(g,t) \in G \times \mathfrak{T}$,  write $\xi_x(g,t,u)$ for the element of $G \times \mathfrak{T}$ obtained by this action. 
    
    For $x \in X$, $g\in G$ and $t \in \mathfrak{T}$ such that $y_x(g) \in \mathcal{E}$, we define $c[x,g,t]\in A^{F_k}$ by \[c[x,g,t](u) = \Phi\bigl(y_x(h)\bigr)(t') \mbox{ for every } u \in F_k, \mbox{ where }(h,t') = \xi_x(g,t,u).\]

    A straightforward computation shows that for $x \in X$, $g,h \in G$ and $t \in \mathfrak{T}$ we have $c[hx,g,t]=c[x,h^{-1}g,t]$.

    Next we define a map $\psi_{\mathcal{E}}\colon \Aut(A^{F_k}) \to \Aut_{\mathcal{E}}( \eta_{\mathcal{E}}(X))$. For $\varphi \in  \Aut(A^{F_k})$, $x \in X$ and $g \in G$ we let \[ \Bigl(\bigl(\psi_{\mathcal{E}}(\varphi)\bigr)(y_x)\Bigr)(g) = \begin{cases}
        \Phi^{-1}\Bigl( \delta(y_x(g)) ,\bigl(\varphi(c[x,g,t]])(\Id_{F_k})\bigr)_{t \in \mathfrak{T}}\Bigr) & \mbox{ if } y_x(g) \in \mathcal{E},\\
        \star & \mbox{ if } y_x(g) = \star.
    \end{cases}  \]

    By~\Cref{lem:egg_exchanges_well_defined} it follows that for every $\varphi \in \Aut(A^{F_k})$ and $x \in X$, then $\bigl(\psi_{\mathcal{E}}(\varphi)\bigr)(y_x) \in \eta_{\mathcal{E}}(X)$. It is clear by definition that $\psi_{\mathcal{E}}(\varphi)$ is continuous, and its $G$-equivariance follows from the fact that $c[hx,g,t]=c[x,h^{-1}g,t]$, thus $\psi_{\mathcal{E}}(\varphi) \in \operatorname{End}(\eta_{\mathcal{E}}(X))$. Furthermore, $\psi_{\mathcal{E}}$ induces an homomorphism from $\Aut(A^{F_k})$ to $\operatorname{End}(\eta_{\mathcal{E}}(X))$, and thus in fact $\psi_{\mathcal{E}}(\varphi) \in \Aut(\eta_{\mathcal{E}}(X))$. Finally, as the position of stars is preserved, each $\psi_{\mathcal{E}}(\varphi)$ is in fact an egg automorphism and thus we conclude that $\psi_{\mathcal{E}}\colon \Aut(A^{F_k}) \to \Aut_{\mathcal{E}}( \eta_{\mathcal{E}}(X))$ is a homomorphism. 
    
    Let us show that $\psi_{\mathcal{E}}$ it is injective. Let $\varphi_1,\varphi_2$ be distinct automorphisms on $\Aut(A^{F_k})$ and $z \in A^{F_k}$ such that $\varphi_1(z)(\Id_{F_k}) \neq \varphi_2(z)(\Id_{F_k})$. For the next part we split the analysis again into the two cases.
    
    In case (1), as $G$ is infinite, there exists a bi-infinite geodesic $(g_i)_{i \in \ZZ}$ (that is, such that for $i,j\in \ZZ$, $|g_i^{-1}g_j|_S=|i-j|$. For $i \in \ZZ$, choose $\gamma_i = g_{(148r+2)i}$ and notice that the collection $ \{\gamma_i B(74r)K \}_{i \in \ZZ}$ is pairwise disjoint as $B(74r)K \subset B(74r+1)$, and moreover, $|\gamma_{i}^{-1}\gamma_{i-1}|=|\gamma_{i}^{-1}\gamma_{i+1}| = 148r+2 \leq (296r+6)N$, thus both $\gamma_{i}^{-1}\gamma_{i-1}$ and $\gamma_{i}^{-1}\gamma_{i+1}$ are in $D$. Through the canonical identification $\ZZ \cong F_1$, we rewrite the collection as $ \{\gamma_u \}_{u \in F_1}$. 

    In case (2), we use~\Cref{lem:fat_whyte} with $m = 74r+1$ and obtain a collection $\{\gamma_u\}_{u \in F_k}$ of elements of $G$ which satisfies that $(\gamma_u B(74r+1))_{u \in F_k}$ is pairwise disjoint, and $|\gamma_{u}^{-1}\gamma_v| \leq (296r+6+d_{F_k}(u,v))N$ for every $u,v \in F_k$. In particular, it follows that $(\gamma_u B(74r)K)_{u \in F_k}$ is pairwise disjoint and that for any $i \in \{1,\dots,k\}$ and generator $a_i$ of $\mathcal{B}$, we have \[|\gamma_{u}^{-1}\gamma_{ua_i}| \leq (296r+7)N \mbox{ and } |\gamma_{u}^{-1}\gamma_{ua_i^{-1}}| \leq (296r+7)N.\]
    Therefore both $\gamma_{u}^{-1}\gamma_{ua_i^{-1}}$ and $\gamma_{u}^{-1}\gamma_{ua_i}$ are in $D$.

    Now we continue with both cases at the same time. Fix some $b \in A$. For $u \in F_k$ we let $(d_u,w_u) \in D^{2k}\times A^{\mathfrak{T}}$ be given by
    \begin{align*}
        d_u & = \bigl( \gamma_{u}^{-1}\gamma_{ua_1^{-1}}, \gamma_{u}^{-1}\gamma_{ua_1}, \dots, \gamma_{u}^{-1}\gamma_{ua_k^{-1}} , \gamma_{u}^{-1}\gamma_{ua_k}\bigr),\\
        w_u(t) & = \begin{cases}
           z(u) & \mbox{if } t = (\top,\dots,\top).\\
           b & \mbox{otherwise.}
        \end{cases}   
    \end{align*}
    
    For $u \in F_k$, set $q_u = \Phi^{-1}(d_u,w_u)$. As $X$ is strongly irreducible and $\{\gamma_u B(74r)K \}_{u \in F_k}$ is pairwise disjoint, there exists $x \in X$ such that $(\gamma_u^{-1}x)|_{B(74r)} = q_u$ for every $u$. By construction, it follows that $c[x,\gamma_{\Id_{F_k}},(\top,\dots,\top)] = z$ and thus \[  \Bigl(\bigl(\psi_{\mathcal{E}}(\varphi_1)\bigr)(y_x)\Bigr)(\gamma_{\Id_{F_k}}) \neq \Bigl(\bigl(\psi_{\mathcal{E}}(\varphi_2)\bigr)(y_x)\Bigr)(\gamma_{\Id_{F_k}}).\]
    Therefore $\psi_{\mathcal{E}}(\varphi_1)\neq \psi_{\mathcal{E}}(\varphi_2)$ which shows that $\psi_{\mathcal{E}}$ is injective. We conclude that $\Aut(A^{F_k})$ embeds into the group of egg automorphism $\Aut_{\mathcal{E}}( \eta_{\mathcal{E}}(X))$ and thus into $\Aut(X)$.\end{proof} 

   \begin{proof}[Proof of~\Cref{thm:EmbeddingF2}]
   Let $K$ be the strong irreducibility constant of $X$. If $G$ is not locally finite, we let $H\leqslant G$ be an infinite finitely generated subgroup which contains $K$. If $G$ is furthermore nonamenable, we take $H\leqslant G$ as a nonamenable finitely generated subgroup which contains $K$ (note that locally amenable groups are amenable \cite[Corollary 4.5.11]{ceccherini-SilbersteinC09}). 

   Let $X|_H = \{ x|_H : x \in X\}$ and notice that it is also non-trivial. By~\Cref{prop:basic-things-SI} we have that $K$ is also a strong irreduciblity constant for $X|_H$, that $X|_H$ has the strong TMP and that $\Aut(X|_H)$ embeds into $\Aut(X)$. The result follows by applying~\Cref{thm:embeddingF2_finitely_gen} to $X|_H$.\end{proof}

\section{Questions and perspectives}\label{sec:questions}

As witnessed by the results of Hochman~\cite{Hochman2010}, our version of Ryan's theorem does not cover all cases where the conclusion holds. 

\begin{question}
Can the hypothesis of strong irreducibility in~\Cref{thm:Ryan} be significantly weakened?
\end{question}

A more specific question is whether the same conclusion holds for transitive SFTs with positive entropy on amenable groups (cf.~\cite[Theorem 1.2]{Hochman2010}). 

There is also in the literature a finitary version of Ryan's theorem by Salo~\cite{salo2019finitaryryan} (see also~\cite{Kopra2020_finitaryRyan}) that states that there exists a finite set $F$ of automorphisms of $\{0,1,2,3\}^{\ZZ}$ such that if $\varphi \in \{0,1,2,3\}^{\ZZ}$ commutes with every element of $F$, then $\varphi$ is a power of the shift. We do now know if similar results hold for arbitrary groups.

\begin{question}
Can a finitary version of Ryan's theorem be established for actions of infinite groups?
\end{question}

All of the results in~\Cref{sec:embedding} are about embedding the automorphism group of a full shift on a free group. A natural question is if this can be extended to strongly irreducible subshifts and to groups which are not free.

\begin{question}
Under which conditions do the automorphism groups of two strongly irreducible subshifts (on the same group) embed into each other?
\end{question}

In particular, it would be interesting to know when the automorphism group of a strongly irreducible $F_k$-subshift embeds into $\Aut(\{0,1\}^{F_k})$. A natural condition in the previous question (at least when the group is residually finite) would be to ask for dense periodic orbits and finite type conditions.

\begin{question}
Let $G$ be an infinite non-locally finite group. Do $\Aut(\{0,1\}^G)$ and $\Aut(\{0,1,2\}^G)$ embed into each other?
\end{question}

The main issue with this question is that the conveyor belt technique cannot be extended naturally beyond (virtually) free groups. We are not aware of alternative techniques to tackle this question on general groups.

Next we wonder whether the hypothesis of strong TMP is truly required for proving~\Cref{thm:EmbeddingF2}.

\begin{question}
Can the strong TMP condition in~\Cref{thm:EmbeddingF2} be removed or weakened?
\end{question}

We suspect that at least in the non-locally finite case the answer might be positive. Indeed, let us say that a countable group $G$ has property $(\heartsuit)$ if there exists $M \geq 1$ such that for every $F\Subset G$, there is $S\Subset G$ with $|S|\leq M$ and $(g_i)_{i \in \ZZ}$ such that the collection $\{g_iF\}_{i \in \ZZ}$ is pairwise disjoint and $g_i^{-1}g_{i+1} \in S$ for all $i \in \ZZ$.

From~\Cref{lem:fat_free_group} it follows that a group $G$ has a torsion free element if and only if it satisfies property $(\heartsuit)$ with $M=1$. It is not hard to see that a torsion group with property $(\heartsuit)$ would satisfy the embedding theorem  without the hypothesis of strong TMP (In the proof of~\Cref{thm:embeddingF2_finitely_gen}, one takes instead as a set of directions $\{1,\dots,M\}$ and interprets them accordingly depending on the marker). However, we do not know of any examples of torsion groups which satisfy property $(\heartsuit)$.

In Theorems~\ref{thm:EmbeddingZ} and~\ref{thm:EmbeddingF2} we spent some energy showing that the automorphism group of a full $F_2$-shift embeds into another automorphism group. However, our efforts would have been rather useless if in fact $\Aut(\{0,1\}^{F_2})$ embeds into $\Aut(\{0,1\}^{\ZZ})$. 

\begin{question}
Does $\Aut(\{0,1\}^{F_2})$ embed into $\Aut(\{0,1\}^{\ZZ})$?
\end{question}

The main problem when tackling this problem is that besides the well-known obstructions that subgroups of $\Aut(\{0,1\}^{\ZZ})$ must be countable, residually finite, and (if finitely generated) their word problem is in $\textbf{co-\texttt{NP}}$~\cite{salo2022conjugacyreversiblecellularautomata}, we don't have any other tools to rule out potential subgroups of $\Aut(\{0,1\}^{\ZZ})$. Clearly the first two conditions are not useful to this endeavor, and we do not know enough about subgroups of $\Aut(\{0,1\}^{F_2})$ to know if the third condition can be applied.

\Addresses
\bibliographystyle{abbrv}
\bibliography{ref}
\end{document}